\documentclass[12pt]{iopart}
\usepackage{iopams}
\usepackage{enumitem}
\usepackage{xcolor}
\usepackage{setstack}
\usepackage{graphicx}
\usepackage[caption=false]{subfig}
\usepackage{booktabs}
\usepackage{url}
\usepackage[numbers]{natbib}

\newcommand{\argmax}{\mathop{\mathrm{arg\:max}}}
\newcommand{\argmin}{\mathop{\mathrm{arg\:min}}}
\newcommand{\coloneqq}{\mathrel{\vcenter{\hbox{$:$}}{=}}}

\newtheorem{thm}{Theorem}[section]
\newtheorem{prop}{Proposition}[section]
\newtheorem{lemma}{Lemma}[section]
\newtheorem{defn}{Definition}[section]
\newenvironment{proof}[1][Proof]{\paragraph{#1:}}{\hfill$\square$}

\begin{document}

\title[Predictive risk estimation for the EM algorithm with Poisson data]{Predictive risk estimation for the Expectation Maximization algorithm with Poisson data}

\author{Paolo Massa and Federico Benvenuto}
\address{Dipartimento di Matematica, Universit\`a di Genova, via Dodecaneso 35, 16146
Genova, Italy}
\ead{\mailto{massa.p@dima.unige.it}, \mailto{benvenuto@dima.unige.it}}

\vspace{10pt}
\begin{indented}
\item[]September 2020
\end{indented}

\begin{abstract}
In this work, we introduce a novel estimator of the predictive risk with Poisson data, when the loss function is the Kullback-Leibler divergence, in order to define a regularization parameter's choice rule for the Expectation Maximization (EM) algorithm.
To this aim, we prove a Poisson counterpart of the Stein's Lemma for Gaussian variables, and from this result we derive the proposed estimator showing its analogies with the  well-known  Stein's  Unbiased  Risk  Estimator valid for a quadratic loss.
We prove that the proposed estimator is asymptotically unbiased with increasing number of measured counts, under certain mild conditions on the regularization method.
We show that these conditions are satisfied by the EM algorithm and then we apply this estimator to select its optimal reconstruction.
We present some numerical tests in the case of image deconvolution, comparing the performances of the proposed estimator with other methods available in the literature, both in the inverse crime and non-inverse crime setting.
\end{abstract}

\vspace{2pc}

\section{Introduction}
The Expectation Maximization (EM) algorithm is a general iterative strategy proposed in \cite{dempster1977maximum} to compute the Maximum Likelihood (ML) solution of a given estimation problem.
EM has been widely applied to Poisson inverse problems, i.e. when measurements are corrupted by Poisson noise: examples arise in medical imaging, astronomical imaging, microscopy and whenever data are collected by means of a counting process \cite{bertero2018inverse,hohage2016inverse,shepp1982maximum, Snyder:93}. 
In the special case of deconvolution, it was known as Richardson-Lucy algorithm \cite{lucy1974iterative,richardson1972bayesian} before the formalization given in \cite{dempster1977maximum}.
It is well known that the EM algorithm for Poisson data converges to the ML solution which is an unreliable solution of the problem as it is corrupted by noise amplification due to the ill-posedness of the problem.
Nonetheless, at the beginning of the EM iterative process the approximated solutions improve, so that finding the optimal number of iterates, i.e. when the algorithm starts diverging from the exact solution of the inverse problem, is a long standing problem which is of great interest for a number of studies.
Such a property is widely known as semiconvergence, or sometimes authors refer to it as self-regularization \cite{munk2010self}.
In \cite{Resmerita_2007} authors have theoretically proven the semiconvergent behaviour of the EM algorithm for noisy data, under some mild assumptions.

Among the many different strategies proposed to stop the EM algorithm for Poisson data, (see e.g. \cite{benvenuto2014regularization,Bertero_2010, bissantz2008statistical,hohage2016inverse,llacer1991use, munk2010self} and references therein), in this work we focus on rules based on predictive risk minimization, that has recently gained some attention in the Gaussian framework \cite{benvenuto2019discrepancy,benvenuto2020parameter, deledalle2014stein,li2020empirical,lucka2018risk}.
The predictive risk is the expectation of the loss function between the the data and the prediction and, in the case of Poisson data, the loss function is the Kullback-Leibler (KL) divergence.
To the best of our knowledge, only two strategies have been proposed in the literature to estimate the predictive risk in the case of Poisson inverse problems, 
each of which relying on a different estimation of the risk function. 
The first approach was proposed in \cite{santos2009new} and it introduces an approximation of the predictive risk based on the Taylor expansion of the KL divergence, without proving its convergence.
The second approach has been more recently proposed in \cite{deledalle2017} and it exploits a statistical estimator valid for the entire class of the generalised linear models (see equations (2.3) and (2.10) in \cite{hudson1978natural}), obtaining an estimator of the predictive risk in the case of Poisson data with an unknown but constant bias, named PUKLA for Poisson Unbiased KL Analysis.
The striking fact is that, although the connection between inverse problems and statistical estimation has been well established (see e.g. \cite{Evans_2002}), it took about forty years before this result was applied to an inverse problem.
Nonetheless, this estimator cannot be  straightforwardly applied to large sized data due to its computational complexity.
In order to overcome this problem, in \cite{deledalle2017} the author proposed an approximation of PUKLA which relies on a combination of a Taylor expansion and a Monte-Carlo integration (see equation (6.2) in  \cite{deledalle2017}).
Also in this case, convergence and statistical properties of this approximation have not been investigated.

In this work, we introduce a novel estimator of the predictive risk with Poisson data, proving that it is asymptotically unbiased when used to choose the optimal number of iterations of the EM algorithm.
In order to do this, we state and prove a more general result, a Poisson counterpart of the Stein's Lemma for Gaussian variables
\cite{stein1981estimation}.
Analogously to the well-known Stein's Unbiased Risk Estimator (SURE), the proposed predictive risk estimator is composed by three parts: the first and the second part of the estimator represent the bias and the variance, respectively, while the third term is the negative asymptotic variance of the estimator.
In line with the Gaussian case, the bias term is given by the empirical loss function, whereas the variance term involves the derivative of the influence (or smoothing) operator and the third term is constant.
The actual difference between the two decomposition is that in the Poisson case a fourth term exists and, under some regularity assumptions, it tends to zero with increasing number of counts, making this estimator asymptotically unbiased.
Operationally, the idea is then to select the iterate of the EM algorithm when this function takes the minimum, i.e. in terms of an estimation of the bias-variance trade-off.
However, we show with some numerical tests in the case of image deconvolution that the performances of all these predictive risk estimators are quite similar: the main advantage of the one proposed in this work is its complete theoretical justification. Indeed, this analysis shed some light also about the above mentioned approximations of the KL based risk (see \cite{deledalle2017, santos2009new}).
The comparison with the Gaussian case becomes straightforward and it is a reference point to identify the role of the terms of the existing estimates.
Moreover, it becomes clear that the statistical convergence of these estimators is not guaranteed as it depends on some regularity conditions of the regularization method.
Finally, such conditions concern the mathematical properties of the regularization method and they can be verified in advance.

The paper is organised as follows.
In section 2 we describe the mathematical formulation of an inverse problem with Poisson data and the EM algorithm viewed as a one-parameter family of estimators.
In section 3 we introduce the risk definition and we give the main results of the paper: the asymptotic Stein's Lemma for Poisson variables and the Poisson Asymptotically Unbiased KL risk estimator.
Then, in section 4 we prove that the EM algorithm satisfies the conditions under which the proposed predictive risk estimator computed along the EM iterates is asymptotically unbiased.
In section 5 we show some numerical tests in the case of image deconvolution, comparing the use of predictive risk estimators and the Poisson discrepancy principle \cite{Bertero_2010}.
Section 6 is devoted to conclusions.
Proofs of the main results are contained in the Appendix.

\section{Mathematical formulation}

The aim of an ill-posed inverse problem in a semi-discrete setting \cite{bertero1985linear,Resmerita_2007} is to estimate a function $u^\ast\in L^2 (\mathbb{R}^d)$, representing an emission density, having random  noisy measurements $Y_1, \dots Y_M$ of the value of $M$ linear functionals that can be thought as the components of $A u^\ast$, where $A \colon L^2(\mathbb{R}^d) \to \mathbb{R}^M$ is defined by 
\begin{equation}
\label{model equation}
(Au)_i = \int_{\Omega} u(s)h_i(s) \,ds
\quad , \quad
\forall ~ u\in L^2 (\mathbb{R}^d) \qquad
\forall ~ i\in\{1, \dots, M\}
\end{equation}
with $\Omega \subseteq \mathbb{R}^d$ a compact set and $h_i\in L^\infty(\Omega)$. Typically, the dimension $d$ is equal to $2$ or $3$.
In the case of Poisson measurements, the additional hypothesis is that
\begin{equation}
    \label{poisson_noise}
Y = (Y_1, \dots, Y_M) \sim \mathcal P( Au^\ast + b )
\end{equation}
where $\mathcal P(\lambda)$ is the Poisson distribution of parameter $\lambda$ with independent components and $b=(b_1, \dots, b_M)$ is a vector whose components $b_i>0$ are the mean values of the detected background emission. In applications, such as microscopy or astronomy, equation \eref{model equation} represents the forward linear operator describing the signal formation by means of the instrument acquisition system: in this case $\Omega$ is the field of view of the imaging system and $h_i$ is the transmission function describing the 
probability that a photon, coming out from the point $s$, reaches the $i$-th detector.
Moreover, the number of counts registered by the $M$ detectors is a realization of $Y_1, \dots, Y_M$. Finally, the function $u^\ast \geq 0$ satisfies the non-negativity hypothesis as it represents the intensity of a photon emission.

\ifdefined\APPLICAZIONI
In many applications, such as astronomical imaging or Positron Emission Tomography (PET), an unknown object is reconstructed from a limited set of photon count measurements. In a semi-discrete setting \cite{bertero1985linear}, the aim is to determine a function $x^\ast\in L^2 (\mathbb{R}^2)$ representing the emission density (therefore $x^\ast \geq 0$) given the number of counts registered by $M$ detectors. The forward linear operator that mimics the instrument acquisition is a map $H \colon L^2(\mathbb{R}^2) \to \mathbb{R}^M$ defined by 
\begin{equation}\label{model equation}
(Hx)_i = \int_{\Omega} x(s)h_i(s) \,ds
\end{equation}
for all $i\in\{1, \dots, M\}$ and for all $x\in L^2(\mathbb{R}^2)$, where $\Omega \subseteq \mathbb{R}^2$ is a compact set representing the field of view of the imaging system and $h_i\in L^\infty(\Omega)$ are the transmission functions describing what is the 
probability that a photon,  coming out from the point $s$, reaches the $i$-th detector.  
\fi

In practical applications a discretisation of \eref{model equation} is required; this means that the domain $\Omega$ is partitioned into $N$ cells with the same area (or volume) $\Delta s$. \Eref{model equation} can be rewritten as
\begin{equation}
(Au)_i = \sum_{j=1}^N x_j h_{i,j} \Delta s ~, 
\end{equation}
where $x_j$ and $h_{i,j}$ are the mean values in the $j$-th cell of the functions $u$ and $h_i$ respectively. 
The unknown function $u^\ast$ becomes then a vector in $\mathbb{R}^N$ and the forward linear operator a matrix which we denote with $x^\ast$ and  $H\in\mathbb{R}^{M\times N}$ respectively; in particular the entries of $H$ are given by
\begin{equation}
H_{ij} = h_{i,j} \Delta s ~.
\end{equation}
As measurements are counting processes drawn from a Poisson random  variable $Y$ with independent components and mean value 
$\lambda = H x^\ast +  b \in\mathbb{R}^M_{++}$, we can assume that the vector of data recorded is a realization $y\in\mathbb{R}^M_+$.
Then, the problem of determining the object from the measurements is the discrete inverse problem given by the equation
\begin{equation}\label{Eq inverse problem}
    y = Hx +  b ~,
\end{equation}
where $x = (x_1, \dots, x_M)$. We assume $H_{ij}>0$ for all $i$, $j$ as in \cite{Resmerita_2007}; we point out that this assumption, very restrictive from a theoretical view point, is not a limitation in practice as we can always numerically approximate a matrix with non-negative components with another one with arbitrarily small positive entries.
For example, denoising applications, for which $H=I$ and $b=0$, can be considered as a limiting case of this hypothesis.

The probability of observing a data $y$ corresponding to an object $x$ is given by
\begin{equation}
p(y ; x) \coloneqq \prod_{i=1}^M \frac{e^{-(Hx + b)_i   } (Hx + b)_i^{y_i}}{y_i!} ~,
\end{equation}
therefore, determining the unknown corresponds to estimating the parameter of a Poisson variable. 
The ML approach consists in finding the value of the parameter $x$ for which the log-likelihood function $\log L_y(x) \coloneqq \log(p(y ; x))$ takes its maximum, i.e. in solving
\begin{equation}\label{max log likelihood}
\argmax_{x\in\mathbb{R}^M_+} \left\{ \log L_y(x) = \sum_{i=1}^M -(Hx + b)_i + y_i \log(Hx + b)_i - \log(y_i !) \right\} ~.
\end{equation}

It is worth noticing that, by changing the sign of the objective function and by subtracting terms that are constant with respect to $x$, it can be readily shown that problem \eref{max log likelihood} is equivalent to
\begin{equation}\label{problem min Dkl}
\argmin_{x\in\mathbb{R}^M_+} D_{\mathrm{KL}}(y, Hx +  b) ~,
\end{equation}
where $D_{\mathrm{KL}}$ is the Kullback-Leibler divergence \cite{bertero2008iterative} defined by
\begin{equation}
D_{\mathrm{KL}}(u, v) \coloneqq \sum_{i=1}^M u_i \log \frac{u_i}{v_i} + v_i - u_i
\end{equation}
for all $u\in\mathbb{R}^M_+$ and for all $v\in\mathbb{R}^M_{++}$.
When data are Poisson distributed, the EM algorithm takes the form
\begin{equation}
x_{k+1} = \psi_y(x_k)
\end{equation}
where
\begin{equation}
\psi_y(x) \coloneqq \frac{x}{H^T 1} H^T \left(\frac{y}{Hx +  b} \right) ~.
\end{equation}
where $1$ represents a vector with entries all equal to one and multiplication and division between vectors are meant component-wise. The usual initialisation is $x_0 = 1$.
This algorithm was presented in \cite{Snyder:93} as a application of the general EM strategy proposed in \cite{dempster1977maximum} to the case of image deconvolution and it is a slight modification of the method introduced by Richardson \cite{richardson1972bayesian} and Lucy \cite{lucy1974iterative} for taking into account the background emission $b$.
Moreover, examples when the operator $H$ is space variant arise both in astronomy \cite{benvenuto2013expectation,massa2019count} and biomedicine \cite{politte1991corrections,shepp1982maximum}.
Following \cite{benvenuto2017study, benvenuto2016robust}, the $k$-th iterate of the EM algorithm can be written in closed form as a function of the data $y$. In particular, the function $R_k \colon \mathbb{R}_+^M \to \mathbb{R}_+^N$ such that $R_k(y) = x_k$, is
\begin{equation}
\label{reg_alg}
R_k(y) \coloneqq (\underbrace{\psi_{y} \circ \ldots \circ \psi_{y}}_{k~{\rm times}}) (x_0) 
\end{equation}
with $R_0(y) \coloneqq x_0$ for all $y\in\mathbb{R}^M_+$.

Due to the ill-posedness of the inverse problem \eref{model equation}, the discrete forward operator $H$ is ill-conditioned and the ML solution is deprived of physical meaning \cite{bertero1998introduction}. 
Moreover, the semiconvergent behaviour of the EM algorithm is a well known property \cite{bertero1998introduction,Resmerita_2007}; in particular for $b=0$, it has been recently proved in \cite{Pouchol_2020} that the solution of \eref{problem min Dkl} (possibly reached for $k=+\infty$) is composed by a few non-zero components that gives rise to spikes in the reconstructed image.
Therefore, the EM algorithm requires regularization in the form of a stopping rule that prevents over-fitting of noisy data. 

\section{Asymptotic Stein's Lemma for Poisson KL risk estimator}

To regularize the EM algorithm by early stopping the iterations, several approaches have been investigated in the literature; in particular there exists a class of methods based on the predictive risk estimate \cite{bardsley2009regularization, deledalle2017, santos2009new}.
In this paper we follow this strategy, motivated by the considerations below. In the inverse problem \eref{Eq inverse problem}, the data $y$ is a realization of the Poisson random variable $Y$ and the $k$-th iterate of the EM algorithm gives rise to an estimator of its mean value $\lambda$, defined by
\begin{equation}
\hat{\lambda}_k(Y) \coloneqq HR_k(Y)+b ~.
\end{equation}
The semiconvergent behaviour of the EM algorithm is reflected in the fact that $D_{\mathrm{KL}}(y, \hat{\lambda}_k(y)) \to 0$ as $k\to +\infty$; therefore, as the number of iterations increases, the estimate $\hat{\lambda}_k(y)$ of $\lambda$ gets more and more biased towards $y$. Actually, we would like to stop the algorithm at the iterate $k^\ast$ such that 
\begin{equation}
k^\ast = \argmin_{k\in\mathbb{N}}D_{\mathrm{KL}}(\lambda, \hat{\lambda}_k(y)) ~,
\end{equation}
but the computation of $D_{\mathrm{KL}}(\lambda, \hat{\lambda}_k(y))$ is unfeasible because it depends on the value of $\lambda$ which is unknown. The idea is to replace this quantity with its expected value with respect to the random variable $Y$ and therefore to find the iterate $k^\ast$ that satisfies
\begin{equation}\label{min pred risk}
k^\ast = \argmin_{k\in\mathbb{N}} \mathbb{E}(D_{\mathrm{KL}}(\lambda, \hat{\lambda}_k(Y)) ~.
\end{equation}
We observe that $\mathbb{E}(D_{\mathrm{KL}}(\lambda, \hat{\lambda}_k(Y))$ is the risk of the estimator $\hat{\lambda}_k(Y)$ when the KL divergence is the loss function; from the point of view of the estimator $R_k(Y)$ of the reconstructed image it is called \textit{predictive risk}. For the sake of clarity, we recall the definition of risk in the statistical framework. It is worth noticing that the definition we adopted is obtained by considering an estimator as decision function \cite{berger2013statistical, Evans_2002}.

\begin{defn}[Risk of an estimator]
Let $\theta\in\Theta$  be a parameter of a random variable $Y$ that takes values in a set $\mathcal{Y}$ and let  $\hat{\theta}\colon \mathcal{Y} \to \mathcal{A}$ be a function such that $\hat{\theta}(Y)$ is an estimator of $\theta$. The risk of the estimator $\hat{\theta}(Y)$, computed with respect to the loss function  $L \colon \Theta \times \mathcal{A} \to [0, +\infty]$, is the quantity
\begin{equation}
r_\theta(\hat{\theta}(Y))\coloneqq \mathbb{E}(L(\theta, \hat{\theta}(Y))) ~.
\end{equation}
\end{defn}
This definition applies to our case by taking $L$ the KL divergence, $Y$ a Poisson random variable and $\theta$ its mean value $\lambda$. Moreover, the role of $\hat{\theta}$ is played at each iteration of the EM algorithm by $\hat{\lambda}_k$. Once again one point out that, in order to compute the value of $\mathbb{E}(D_{\mathrm{KL}}(\lambda, \hat{\lambda}_k(Y))$, it is necessary to know $\lambda$; actually, as we will show in the following, it is possible to obtain estimators of $\mathbb{E}(D_{\mathrm{KL}}(\lambda, \hat{\lambda}_k(Y))$ that do not depend on $\lambda$, therefore we will determine the iterate $k^\ast$ by replacing the risk in \eref{min pred risk} with an estimator of its.

Unbiased risk estimation has been widely studied in the past decades; in particular, the most famous case is the one concerning normal random variables and the square $\ell^2$ norm as loss. In \cite{stein1981estimation} Stein proved that, given $Y \sim \mathcal{N}(\mu, \sigma^2 I)$ an $M$-dimensional normal random variable and $\hat{\mu}(Y)$ an estimator of $\mu$, under suitable conditions on $\hat{\mu}$ it holds true that the risk $\mathbb{E}\left(\Vert\hat{\mu}(Y) - \mu\Vert^2 \right)$ can be unbiasedly estimated by the Stein's Unbiased Risk Estimator (SURE)
\begin{equation}\label{SURE for gauss}
\mathrm{SURE}(Y) \coloneqq \Vert\hat{\mu}(Y) - Y \Vert^2 + 2\sigma^2 \nabla\cdot \hat{\mu}\left(Y\right) -M\sigma^2 ~.
\end{equation}
The derivation of \eref{SURE for gauss} is essentially based on a result, named Stein's Lemma, which states that, for all $i\in\{1, \dots, M\}$
\begin{equation}\label{Stein's lemma}
\mathbb{E}\left((Y_i - \mu_i)f(Y) \right) = \mathbb{E}\left( \sigma^2 \partial_i f(Y) \right) ~,
\end{equation}
if $f\colon \mathbb{R}^M \to \mathbb{R}$ satisfies some mild conditions. In the case of an $M$-dimensional Poisson random variable  $Y\sim \mathcal{P}(\lambda)$ with independent entries an analogous formulation of \eref{SURE for gauss} was given \cite{deledalle2017, hudson1978natural}. If $\hat{\lambda}(Y)$ is an estimator of $\lambda$, then $\mathbb{E}(\Vert \hat{\lambda}(Y) - \lambda \Vert^2 )$ can be estimated by the Poisson Unbiased Risk Estimator (PURE)
\begin{equation}
\mathrm{PURE}(Y) \coloneqq \Vert \hat{\lambda}(Y)\Vert^2 -2 Y\cdot \hat{\lambda}_\downarrow(Y) + Y \cdot (Y-1) ~,
\end{equation}
where $\hat{\lambda}_\downarrow(Y)_i \coloneqq \hat{\lambda}(Y - e_i)_i$ for all $i\in\{1, \dots, M\}$ and $e_i$ is the $i$-th element of the canonical basis. As explained in \cite{deledalle2017}, considering the square $\ell^2$ norm as loss for the risk of $\hat{\lambda}$ is not very effective in practice because it does not take into account the heteroscedasticity of the Poisson random variable.
Therefore the KL divergence was considered as loss function and it was proved that the risk $\mathbb{E}(D_{\mathrm{KL}}(\lambda, \hat{\lambda}(Y)))$ can be unbiasedly estimated by the Poisson Unbiased KL Analysis (PUKLA) estimator
\begin{equation}\label{PUKLA}
\mathrm{PUKLA}(Y) \coloneqq \Vert \hat{\lambda}(Y) \Vert_1 - Y \cdot \log\hat{\lambda}_\downarrow(Y) ~,
\end{equation}
where the logarithm is applied component-wise. The computation of PUKLA is very expensive in practice as it requires $M$ times the evaluation of $\hat{\lambda}$ (in the image reconstruction problem of interest it means that $M$ reconstructions are needed). Moreover, from a theoretical point of view, it requires that the function $\hat{\lambda}$ is defined on $[-1, +\infty[^M$ and this is not the case for $\hat{\lambda}_k$.
Finally, the PUKLA estimator is not truly unbiased: it does not provide an absolute quantification of the risk as it is defined up to an additive constant. 
In \cite{santos2009new} the authors introduce a different estimator of the risk with KL loss, defined by
\begin{equation}
\fl \mathrm{REKL}(Y) \coloneqq \sum_{i=1}^M  \hat{\lambda}(Y)_i - Y_i \log \hat{\lambda}(Y)_i + \frac{M Y_i \eta_i}{2\varepsilon \Vert \eta \Vert^2} \left( \log \hat{\lambda}(Y + \varepsilon \eta)_i - \log \hat{\lambda}(Y - \varepsilon \eta)_i
 \right)
\end{equation}
where $\eta \sim \mathcal{N}(0, I)$ is an $M$-dimensional normal random variable and $\varepsilon$ is a small positive constant. 
This estimator is derived from a Taylor expansion of the KL divergence and the second term  comes from a Monte-Carlo approximation involving the trace of a Jacobian matrix. 
Although very effective in practice, it has no rigorous theoretical justification, as it is not proven that the Taylor's remainder is negligible, at least in some asymptotic sense. 
Moreover, as the PUKLA estimator, it is defined up to an unknown additive constant.
Other estimators of the predictive risk, based on the application of the Generalized Cross Validation (GCV) and the Unbiased Predictive Risk Estimator (UPRE) to a Gaussian approximation of the KL divergence, have been presented in \cite{bardsley2009regularization}

In this section we propose a new estimator of $\mathbb{E}(D_{\mathrm{KL}}(\lambda, \hat{\lambda}(Y)))$ that is analogous to SURE \eref{SURE for gauss} when Poisson data are considered. 
We also give some mild conditions under which the proposed estimator is unbiased in an asymptotic regime, i.e. when every component of the mean $\lambda$ is sufficiently large. 
Moreover, differently from the previous ones, the proposed estimator is not defined up to an additive constant, but it estimates the actual value of the risk.
First of all we give an asymptotic Stein's Lemma for Poisson random variables. 
The proof of this result is given in the Appendix.
\begin{lemma}[Asymptotic Stein's Lemma for Poisson random variables]\label{Stein's lemma (Poisson)}
Let $Y$ be a multivariate Poisson random variable with mean value $\lambda \in \mathbb{R}_{++}^M$ and independent components, $\mathcal{C}\subset\mathbb{R}^M_{++}$ a closed convex cone containing $\lambda$ and $f\colon \mathbb{R}^M_{+} \to \mathbb{R}$ a function such that:
\begin{enumerate}[label=(\roman*)]
\item
$\mathbb{E} \left( \vert Y_i f(Y) \vert \right) < +\infty$ for all $i\in\{1, \dots, M\}$;
\item $f\in C^{2}(\mathbb{R}^M_+)$;
\item for all $j\in \{ 1, \dots, M \}$ there exist $c_{j}$, $\varepsilon_j>0$ such that 
\begin{equation}
\vert \partial_j f(y) \vert \leq \frac{c_j}{\Vert y \Vert + \varepsilon_j }
\end{equation}
for all $y\in\mathbb{R}_+^M$;
\item for all $k$, $l\in \{ 1, \dots, M\}$ there exist $\varepsilon_{k,l}$, $c_{k,l}>0$ such that
\begin{equation}
\left\vert \partial_k\partial_l f(y) \right\vert \leq \frac{c_{k,l}}{\Vert y \Vert^2 + \varepsilon_{k,l}}
\end{equation}
for all $y\in\mathbb{R}_+^M$.\label{Stein's Lemma hp iv} 
\end{enumerate}
Then, for all $i\in\{1, \dots, M \}$
\begin{equation}
\label{thesis stein's lemma poisson}
\mathbb{E} \left(\left( Y_i - \lambda_i \right) f\left( Y \right)\right) = \mathbb{E} \left(Y_i \partial_i f (Y) \right) + O \left(\Vert \lambda \Vert^{-1/2} \right)
\end{equation}
as $\Vert \lambda \Vert \to +\infty$ in $\mathcal{C}$.
\end{lemma}

The analogy between \eref{Stein's lemma} and \eref{thesis stein's lemma poisson} is in the fact that the left-hand side is the same, while in the right-hand side the term $\sigma^2$, which is the known variance of $Y_i$ in the normal case, is replaced with $Y_i$ that is an unbiased estimator of the variance in the Poisson case. 
We point out that the introduction of the cone $\mathcal{C}$ in \eref{thesis stein's lemma poisson} is just a technical fact needed by the proof in order to have a result which is uniform in the variable $\lambda$ while keeping the property that every component of $\lambda$ tends to $+\infty$ as the norm $\Vert \lambda \Vert$ does.

Lemma \ref{Stein's lemma (Poisson)} allow us to derive an estimator of the risk with KL loss for Poisson variates. The proof is contained in the Appendix.
\begin{thm}[Poisson Asymptotically Unbiased KL risk estimator]\label{KL risk theorem}
Let $Y$ be a multivariate Poisson random variable with mean value $\lambda \in \mathbb{R}_{++}^M$ and independent components and let $\hat{\lambda}(Y)$ be an estimator of $\lambda$ taking values in $\mathbb{R}^M_{++}$. Let us assume that every component of the function $\log\hat{\lambda}\colon \mathbb{R}_+^M \to \mathbb{R}^M$ satisfies the hypotesis $(i)$-$(iv)$ of Lemma \ref{Stein's lemma (Poisson)}. Then, 
\begin{equation}\label{Kullback-Leibler risk estimator}
\mathrm{PAUKL}\left( Y \right) \coloneqq D_{\mathrm{KL}} ( Y, \hat{\lambda}(Y) ) + (Y \nabla) \cdot \log\hat{\lambda}(Y) - \frac{M}{2} ~,
\end{equation}
where $Y \nabla \coloneqq (Y_1 \partial_1, \ldots, Y_M \partial_M)$, is an asymptotically unbiased estimator of the risk $\mathbb{E}(D_{\mathrm{KL}}(\lambda, \hat{\lambda}(Y))$
in the sense that, if $\mathcal{C} \subset \mathbb{R}^M_{++}$ is a closed cone containing $\lambda$, we have
\begin{equation}
\mathbb{E}(D_{\mathrm{KL}}(\lambda, \hat{\lambda}(Y)) = \mathbb{E} \left( \mathrm{PAUKL}\left( Y \right) \right) + O \left(\Vert \lambda \Vert^{-1/2}\right) ~.
\end{equation}
as $\Vert \lambda \Vert \to +\infty$ in $\mathcal{C}$.
\end{thm}

We remark the similarities between SURE and PAUKL: in both the estimators the first term represents the discrepancy between the observed data and the currently estimated mean value, the second term is related to the divergence of the estimator and the third term is the negative expected value between the random variable $Y$ and its mean value.
From a practical point of view, using \eref{Kullback-Leibler risk estimator} instead of \eref{PUKLA} for estimating the KL risk can be very convenient as the computation of \eref{Kullback-Leibler risk estimator} does not require multiple evaluations of $\hat{\lambda}$; on the other hand, if the analytical expression of $(Y \nabla) \log \hat{\lambda}$ is not available or in the case it is recursively defined as in the case of image reconstruction by means of the EM algorithm, the computation of this quantity can be very expensive or even unfeasible. 
We will show in section \ref{numerical simulations} how to obtain an accurate estimate of $(Y\nabla) \cdot \log \hat{\lambda}$ with a Monte-Carlo method.

\section{Predictive risk for the EM algorithm for Poisson data}\label{section: EM properties}
In this section we prove that it is possible to use \eref{Kullback-Leibler risk estimator} for estimating the predictive risk of the $k$-th iterate of the EM algorithm; in particular, we will show that $\log\hat{\lambda}_k$ satisfies the hypothesis of Theorem \ref{KL risk theorem}. Basically, what is needed is that every component of $\log\hat{\lambda}_k$ is smooth and that its first and second order partial derivatives are bounded by a suitable function that decades sufficiently fast at infinity. 
We will prove that these requirements are satisfied thanks to the fact every component of $R_k$ defined in \eref{reg_alg} is a rational function of the data $y$, in which the difference between the degrees of the polynomials at numerator and denominator is 1.

In the following we will define \textit{homogeneous polynomial} of degree $n$ a polynomial $p^{(n)}$ of the form
\begin{equation}
p^{(n)}(y) = \sum_{\alpha\in S(n)} c_\alpha y_1^{\alpha_1} \cdots y_M^{\alpha_M}
\end{equation}
where $S( n )\coloneqq \{\alpha\in\mathbb{N}^M\colon \,\alpha_1 + \dots + \alpha_M = n\}$ and $c_\alpha\in\mathbb{R}$; we will also call $p^{(n)}$ \textit{complete} if $c_\alpha \neq 0$ for all $\alpha\in S(n)$.
\begin{prop}\label{Prop: propr Rk}
The $k$-th iterate of the EM algorithm for Poisson data is a rational function, i.e. for all $j\in\{1, \dots , N \}$ we can write
\begin{equation}
(R_k(y))_j = \frac{p^{(d+1)}_{k,j}(y) + \dots + p^{(0)}_{k,j} }{q^{(d)}_{k,j}(y) + \dots + q^{(0)}_{k,j}}
\end{equation}
with $d\in \mathbb{N}$, where:
\begin{enumerate}[label=(\roman*)]
\item $p^{(n)}_{k,j}$ and $q^{(m)}_{k,j}$ are homogeneous polynomials with nonnegative coefficients for all $n\in \{0,\dots, d+1\}$ and for all $m\in \{0, \dots, d\}$; \label{Prop rational function i}
\item $p^{(d+1)}_{k,j}$, $q^{(d)}_{k,j}$ are complete and $q^{(0)}_{k,j} > 0$. \label{Prop rational function ii}
\end{enumerate}
\end{prop}

Thanks to this result, it is possible to show that $\log \hat{\lambda}_k$ satisfies the hypothesis of Theorem \ref{KL risk theorem}.

\begin{prop}\label{Prop: propr log lambda}
Let us fix the $k$-th iterate of the EM algorithm. For all $i\in\{1, \dots, M\}$ the following hold true:
\begin{enumerate}
\item $(\log \hat{\lambda}_k )_i\in C^{\infty}(\mathbb{R}^M_+)$;\label{Lemma propr log: i}
\item for all $j\in \{ 1, \dots, M \}$ there exist $c_{j}$, $\varepsilon_j>0$ such that 
\begin{equation}
\vert \partial_j (\log \hat{\lambda}_k (y) )_i \vert \leq \frac{c_j}{\Vert y \Vert + \varepsilon_j }
\end{equation}
for all $y\in\mathbb{R}_+^M$;\label{Lemma propr log: ii}
\item for all $l$, $r\in \{ 1, \dots, M\}$ there exist $\varepsilon_{l,r}$, $c_{l, r}>0$ such that
\begin{equation}
\left\vert \partial_l\partial_r (\log \hat{\lambda}_k (y))_i \right\vert \leq \frac{c_{l,r}}{\Vert y \Vert^2 + \varepsilon_{l,r}}
\end{equation}
for all $y\in\mathbb{R}_+^M$.\label{Lemma propr log: iii}
\end{enumerate}
\end{prop}
The proofs are given in the 
Appendix.

\section{Numerical simulations}\label{numerical simulations}
We perform our numerical tests in the case of deconvolution of astronomical images. In particular, we report the results obtained on two $256 \times 256$ images from the Hubble Space Telescope: the nebula NGC7027 (\url{http://www.oasis.unimore.it/site/home/software.html}) and the Horsehead Nebula (\url{http://www.astropy.org/astropy-data}). 
At every iteration of EM we compute the PAUKL estimator of the predictive risk and we stop the algorithm once its minimum is reached.
We compare the performances of PAUKL with the ones of PUKLA, REKL and of the discrepancy principle for Poisson data.

For applying the PAUKL estimator we need to compute the term $(y\nabla) \cdot \log \hat{\lambda}_k(y)$ in \eref{Kullback-Leibler risk estimator}.
In the case of the EM algorithm, it is possible to obtain a recursive expression of this quantity that at every iteration depends only on the previous one, but the computational burden of the operations required is too high in practice.
In our experiments we decide to approximate it by using a Monte-Carlo method derived in \cite{ramani2008monte}. Exploiting the fact that
 \begin{equation}
 (y\nabla) \cdot \log \hat{\lambda}_k(y) = \lim_{\varepsilon \to 0^+}\mathbb{E} \left((y \eta) \cdot \left(\frac{\log \hat{\lambda}_k(y + \varepsilon \eta) - \log \hat{\lambda}_k(y)}{\varepsilon} \right)\right)
 \end{equation}
 where $\eta\sim \mathcal{N}(0, I)$ is an $M$-dimensional normal random variable and the mean value is computed with respect to $\eta$, we estimate
 \begin{equation}\label{Monte-Carlo approx}
 (y\nabla) \cdot \log \hat{\lambda}_k(y) \approx (y \eta) \cdot \left(\frac{\log \hat{\lambda}_k(y + \varepsilon \eta) - \log \hat{\lambda}_k(y)}{\varepsilon} \right)
 \end{equation}
 with $\varepsilon$ fixed to $10^{-3}$. This estimation has the computational cost of two reconstructions and it is therefore very efficient on large-scale problems. Similarly, in order to be able to apply the PUKLA estimator without performing $M$ reconstructions, we decide to use the approximation given by equation (6.2) in \cite{deledalle2017}:
 \begin{equation}
y \cdot \log\hat{\lambda}_{k\downarrow}(y) \approx y \cdot \log\hat{\lambda}_k (y) - (y \zeta) \cdot (\nabla \log\hat{\lambda}_k (y) \zeta)
 \end{equation}
 where $\zeta = (\zeta_1, \dots, \zeta_M)$ and $\zeta_i$ is a Bernoulli variable with parameter $p = 1/2$ taking values in $\{-1, 1\}$ . This equation reveals an idea of \eref{Kullback-Leibler risk estimator}, but this is the first time that a rigorous proof is given. Following what done in \eref{Monte-Carlo approx}, we use the difference quotient to approximate the derivative, obtaining
\begin{equation}
y \cdot \log\hat{\lambda}_{k\downarrow}(y) \approx y \cdot \log\hat{\lambda}_k (y) - (y \zeta) \cdot \left( \frac{\log \hat{\lambda}_k(y + \varepsilon \zeta) - \log \hat{\lambda}_k(y)}{\varepsilon} \right)
\end{equation}
with $\varepsilon$ fixed to $10^{-3}$. Thanks to these considerations, it is possible to compute an approximated version of PUKLA by performing just two reconstructions instead of $M$.

We point out that in the deconvolution inverse problem the entries of the matrix operator $H$ do not satisfy the hypothesis to be strictly positive. This requirement is heavily used to prove the theoretical results, but in our experience we noticed that it is not necessary in practice.

\subsection{Inverse crime setting}\label{inverse crime setting}
We generate the data $y$ as a realization of a Poisson variable $Y \sim \mathcal{P}(Hx^\ast + b)$, where $x^\ast$ is the image of the nebula, $H$ is the convolution operator with a Gaussian Point Spread Function (PSF) with a standard deviation $\sigma = 3$ expressed in pixel unit, and $b$ is a constant background whose components are equal to $100$ for the nebula NGC7027 and equal to $10$ for the Horse Head nebula . For both images we consider three different levels of statistic, i.e. we rescale $x^\ast$ so that the sum of the pixel values is equal to $10^7$ in the low statistic case, $10^8$ in the medium statistic case and $10^9$ in the high statistic case.

We reconstruct the images with the EM algorithm from a single realization of the data $y$ for each level of statistic. At every iteration $k$ of EM we compute the estimators of the predictive risk (PAUKL, REKL and PUKLA) and the discrepancy with the KL divergence, i.e.
\begin{equation}
d_{\mathrm{KL}}(k, y) \coloneqq D_{\mathrm{KL}}(y, \hat{\lambda}_k(y)) ~.
\end{equation}
We stop the iterations when the minimum of the estimators of the risk is reached or when the Poisson discrepancy principle is satisfied, that is once
\begin{equation}
d_{\mathrm{KL}}(k, y) < \frac{M}{2} ~.
\end{equation}
For the sake of brevity, we report the reconstructions obtained just in the low statistic case for the nebula NGC7027 (see figure \ref{fig:ngc7027 low stat}) and in the high statistic case for the Horse Head nebula (see figure \ref{fig:horsehead high stat}). We do not show the images obtained with PUKLA and REKL as they are too similar to the ones obtained with PAUKL.

\begin{figure}[ht]
\centering
\subfloat[Ground truth]{\includegraphics[width= .24\textwidth]{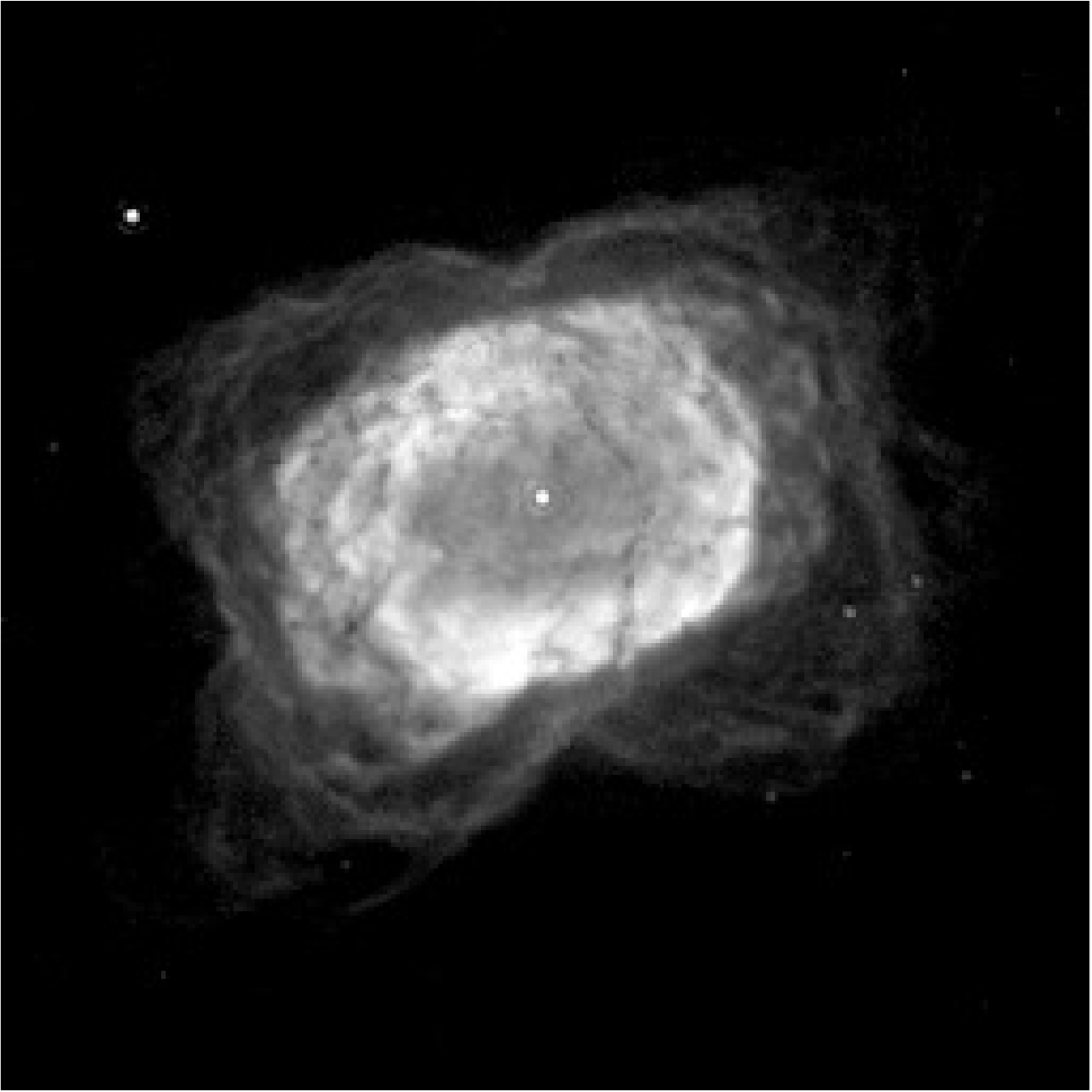}}
\hspace{0.1pt}
\subfloat[Blurred and noisy]{\includegraphics[width= .24\textwidth]{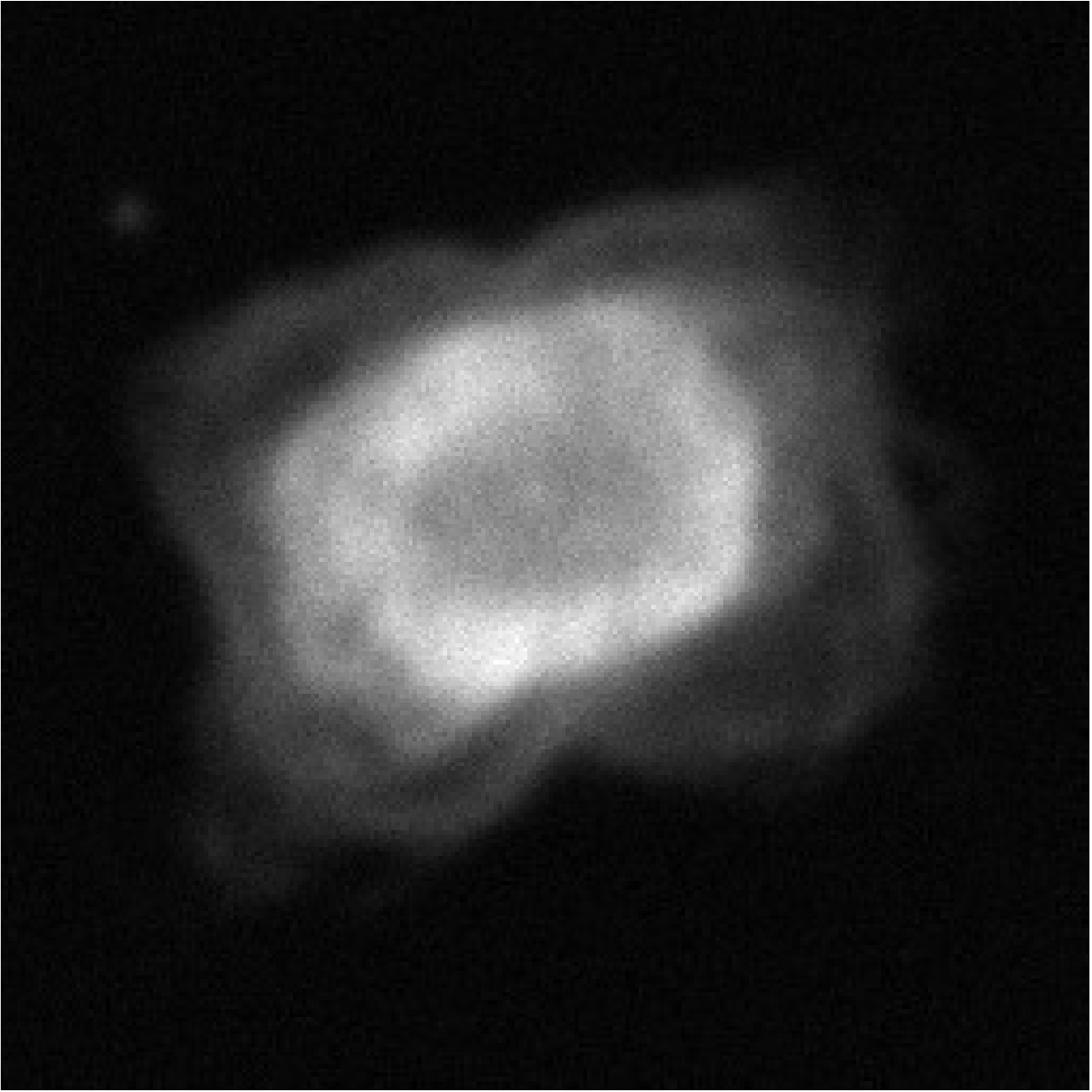}}
\hspace{0.1pt}
\subfloat[PAUKL]{\includegraphics[width= .24\textwidth]{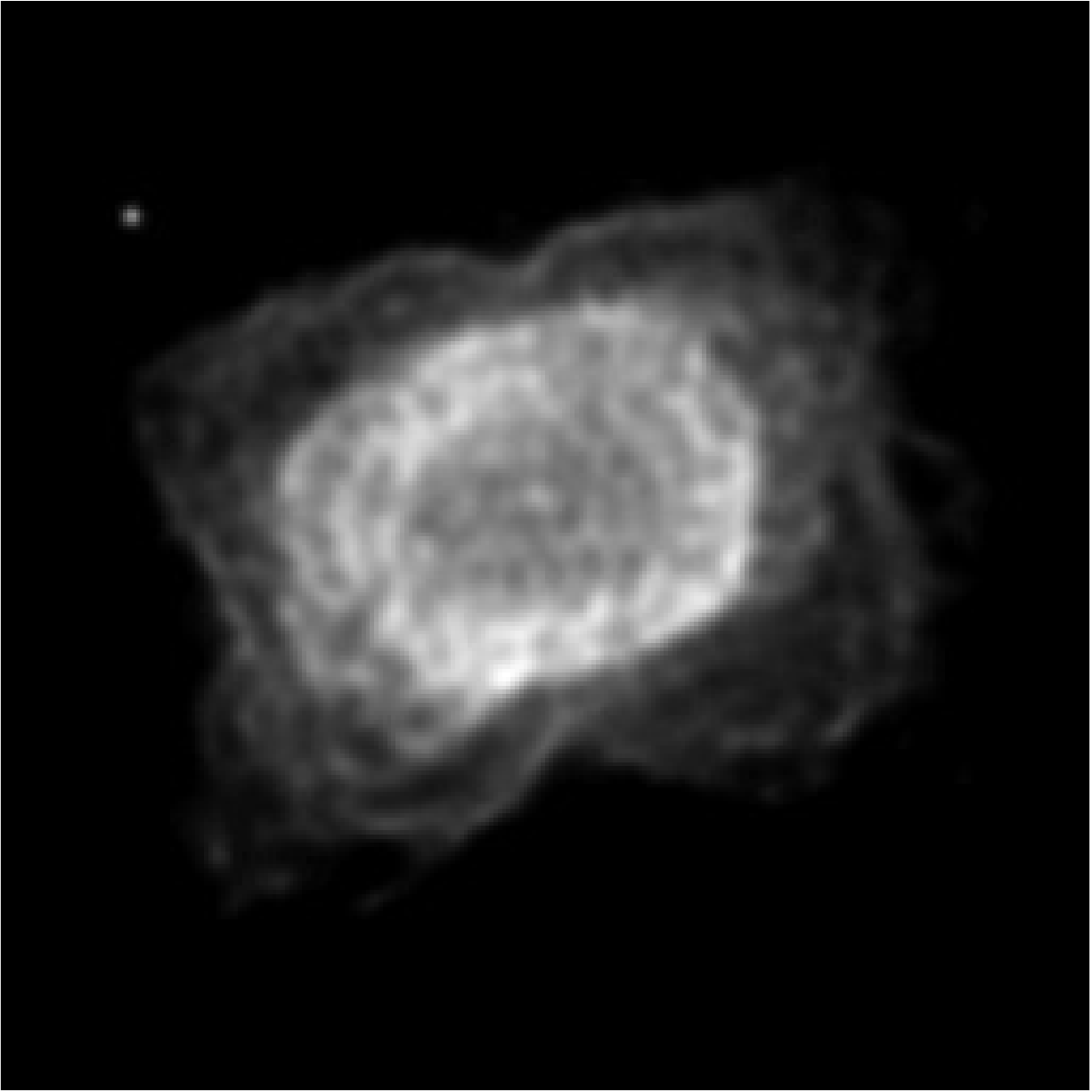}}
\hspace{0.1pt}
\subfloat[Discrepancy, principle][\centering Discrepancy \par principle]{\includegraphics[width= .24\textwidth]{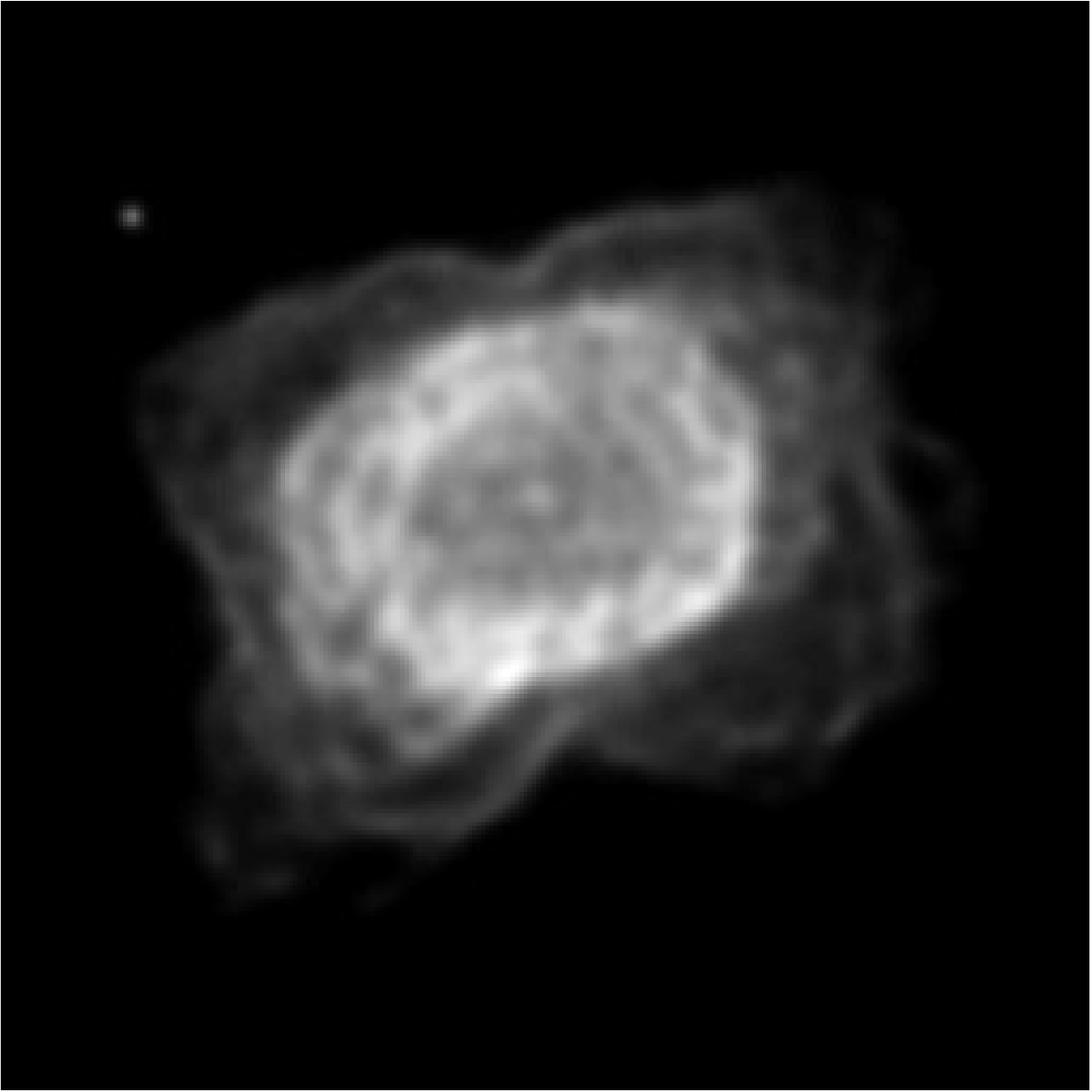}}

\caption{Reconstructions of the nebula NGC7027 in the low statistic case (total flux equal to $10^7$). From left to right: ground truth image, blurred and noisy data, reconstruction obtained with PAUKL and reconstruction obtained with the Poisson discrepancy principle.}
\label{fig:ngc7027 low stat}
\end{figure}

\begin{figure}[t]
\centering
\subfloat[Ground truth]{\includegraphics[width= .24\textwidth]{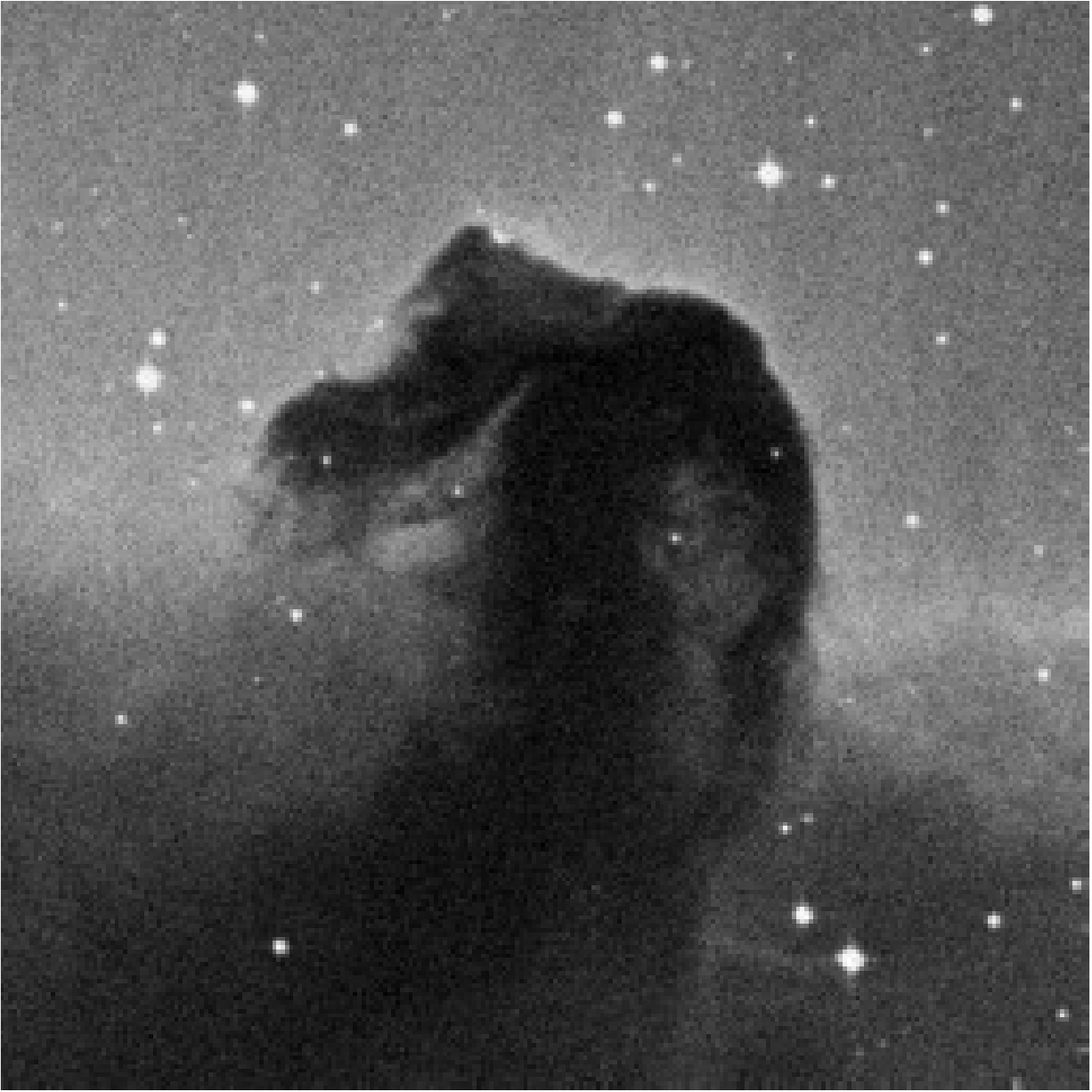}}
\hspace{0.1pt}
\subfloat[Blurred and noisy]{\includegraphics[width= .24\textwidth]{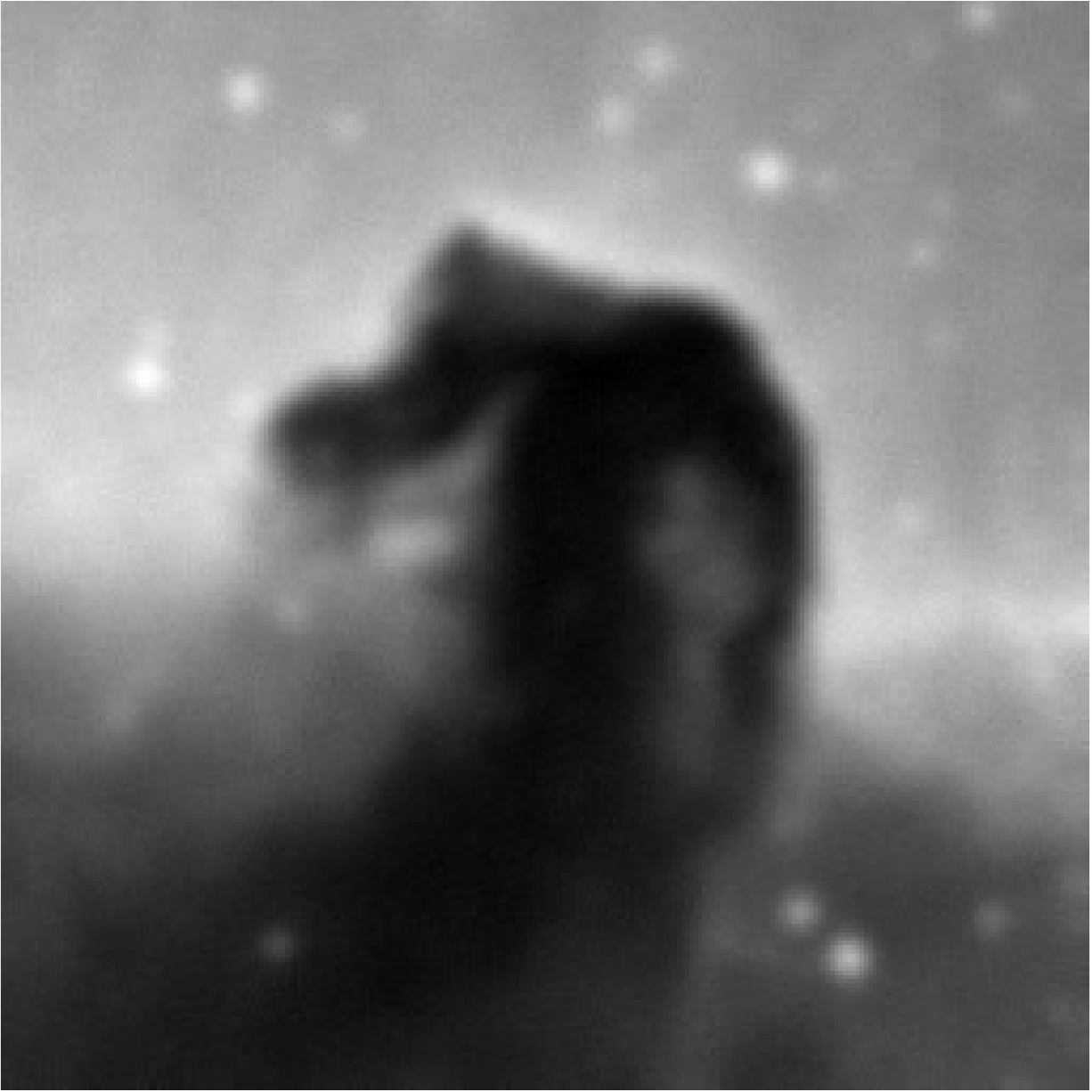}}
\hspace{0.1pt}
\subfloat[PAUKL]{\includegraphics[width= .24\textwidth]{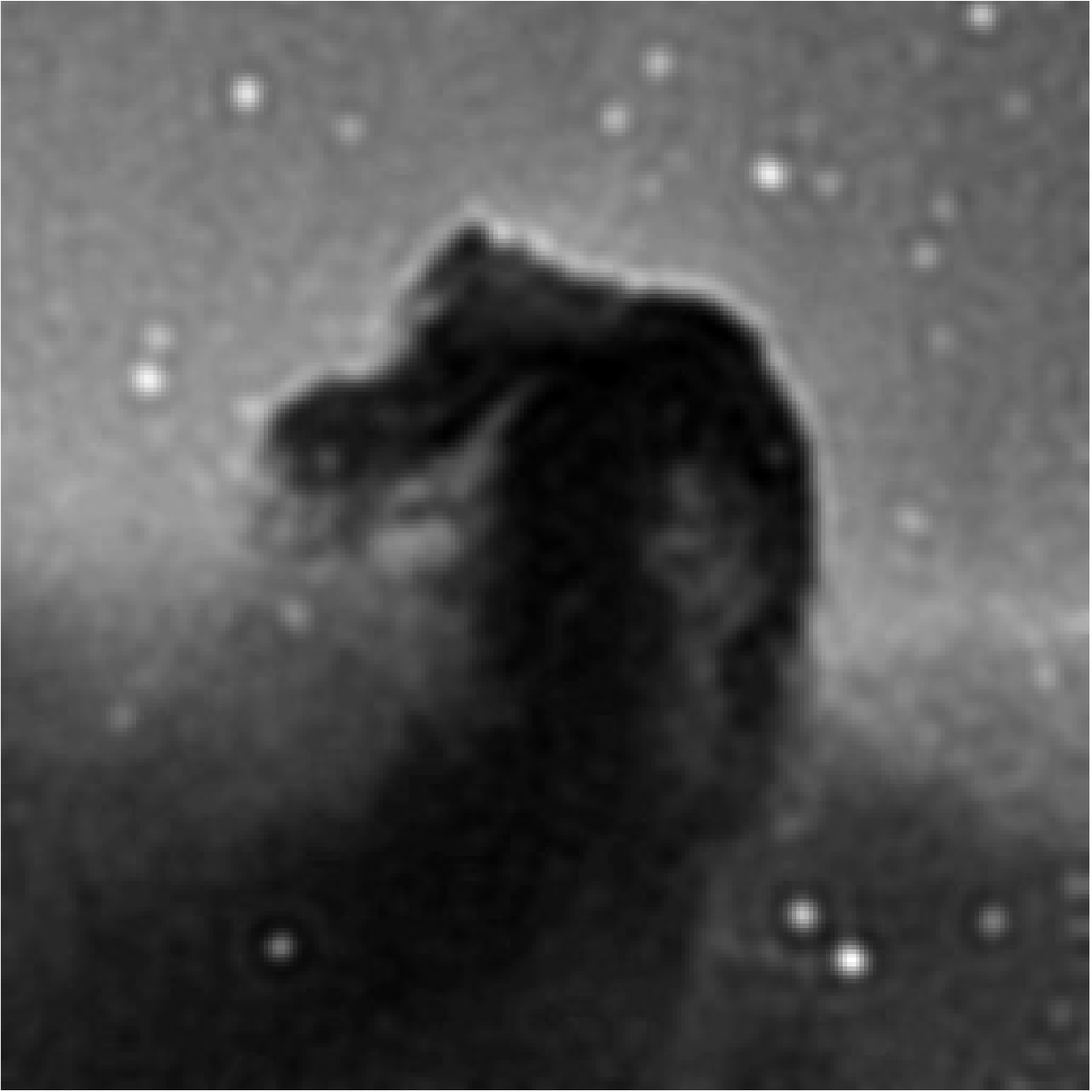}}
\hspace{0.1pt}
\subfloat[Discrepancy, principle][\centering Discrepancy \par principle]{\includegraphics[width= .24\textwidth]{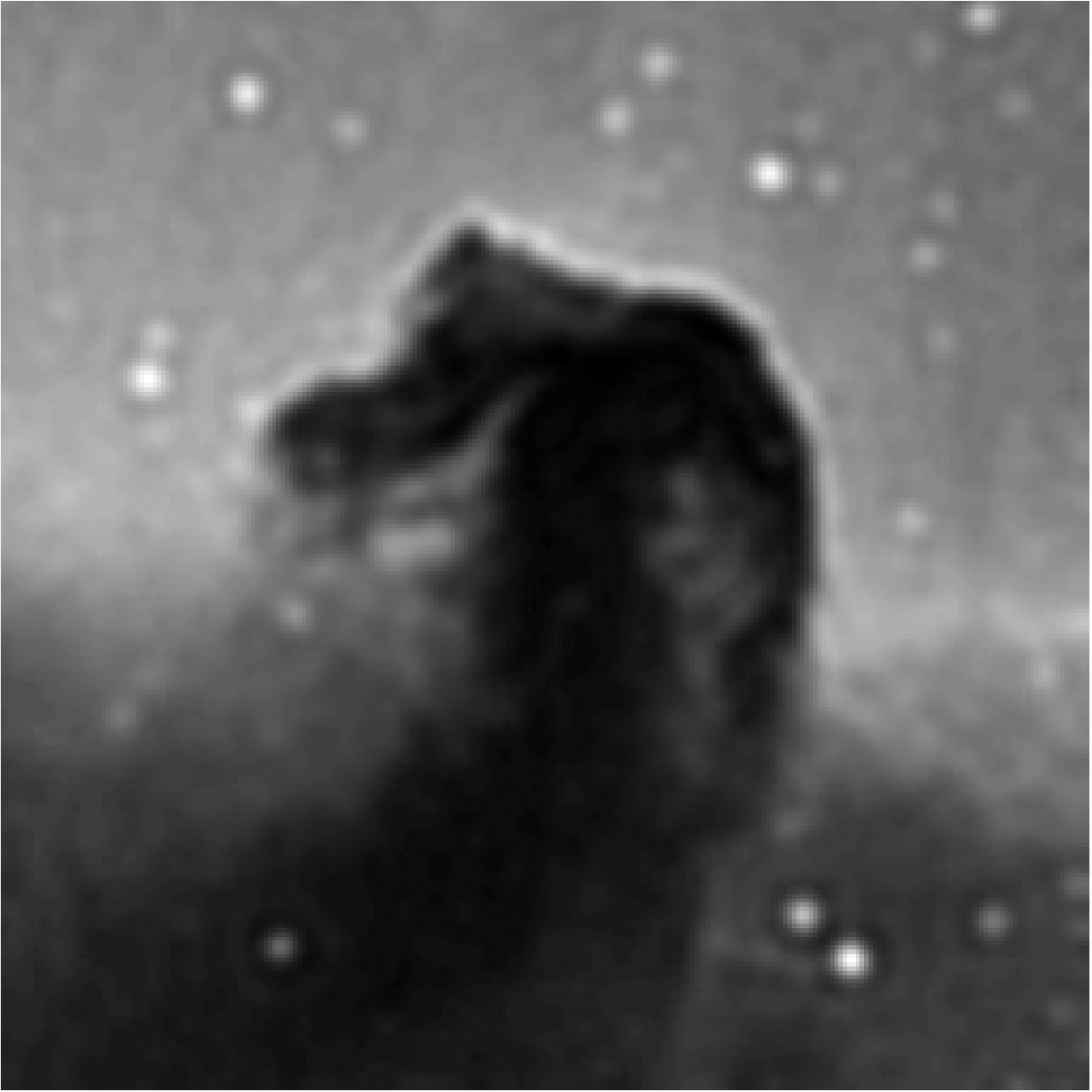}}

\caption{Reconstructions of the Horse Head nebula in the high statistic case (total flux equal to $10^9$). From left to right: ground truth image, blurred and noisy data, reconstruction obtained with PAUKL and reconstruction obtained with the  Poisson discrepancy principle.}
\label{fig:horsehead high stat}
\end{figure}

In order to compare the different estimates of the predictive risk, we perform reconstructions for $25$ realizations of Poisson noise $y_1, \dots, y_{25}$. Defined the Predictive Error (PE) at every iteration as
\begin{equation}
\mathrm{PE}(k,\lambda, y) \coloneqq D_{\mathrm{KL}}(\lambda, \hat{\lambda}_k(y)) ~,
\end{equation}
we compute what we call Sample Predictive Risk (SPR), i.e.
\begin{equation}
\mathrm{SPR}(k, \lambda) \coloneqq \frac{1}{n}\sum_{i=1}^n \mathrm{PE}(k,\lambda, y_i) ~.
\end{equation}
where $n = 25$. This quantity represents the most accurate estimate of the predictive risk that we can have in the simulation tests, as an analytical computation of $\mathbb{E}( D_{\mathrm{KL}}(\lambda, \hat{\lambda}_k(Y)) )$ is not feasible even if $\lambda$ is known. At each iteration we also calculate 
\begin{equation}
\mathrm{PDP}(k, y) \coloneqq \left\vert d_{\mathrm{KL}}(k, y) - \frac{M}{2} \right\vert ~.
\end{equation}
This quantity is related to the Poisson discrepancy principle as the minimum of PDP is reached at the iteration at which the principle is satisfied.
We report in figure \ref{fig:nebula risk medium stat}(a) the SPR and the mean values of PAUKL, REKL, PUKLA and PDP on $25$ realizations of the data for the nebula NGC7027 in the medium statistic case. 
Figure \ref{fig:nebula risk medium stat}(b) contains instead a zoom of figure \ref{fig:nebula risk medium stat}(a) in correspondence to the minimum of the estimators of the predictive risk in order to make the differences between them more visible.
We point out that the quantities reported in figures \ref{fig:nebula risk medium stat}(a) and \ref{fig:nebula risk medium stat}(b) are not directly comparable and therefore each one of them has been translated by adding a constant value.

We have already remarked that PAUKL, differently from REKL and PUKLA, is able to approximate the actual value of the predictive risk. In order to assess this ability, in figure \ref{fig:nebula risk medium stat}(c) we show the SPR and the mean value of PAUKL on 25 realizations of Poisson noise for the image of nebula NGC7027 in the medium statistic case. The shaded area is in correspondence to the standard deviation of the proposed estimator.

We empirically study the problem to understand if the minimum predictive risk is a valid criterion for obtaining good reconstructions. Denoted $\mathrm{err}_{\ell^2}$ and $\mathrm{err}_{\mathrm{KL}}$ the reconstruction errors at every iteration computed with the $\ell^2$ norm and with the KL divergence respectively, i.e.
\begin{equation}
\fl\mathrm{err}_{\ell^2}(k, x^\ast, y) \coloneqq \Vert x^\ast - R_k(y) \Vert ~, \qquad \mathrm{err}_{\mathrm{KL}}(k, x^\ast, y) \coloneqq D_{\mathrm{KL}} (x^\ast, R_k(y)) ~,
\end{equation}
we compute the empirical risks
\begin{equation}
\fl\mathrm{ER}_{\ell^2}(k, x^\ast) \coloneqq \frac{1}{n} \sum_{i=1}^n \mathrm{err}_{\ell^2}(k, x^\ast, y_i) ~, \qquad \mathrm{ER}_{\mathrm{KL}}(k, x^\ast) \coloneqq \frac{1}{n} \sum_{i=1}^n \mathrm{err}_{\mathrm{KL}}(k, x^\ast, y_i) ~,
\end{equation}
with $n=25$, in the case of nebula NGC7027 with medium statistic. We plot them together with the SPR in figure \ref{fig:nebula risk medium stat}(d).
Again, a translation is applied in order to make the lines comparable. 

Finally, we perform $25$ reconstructions with the EM algorithm, for both images and for each level of statistic, from different realization of Poisson noise in the data.
For each reconstruction we compute the number of iterations at which the minimum of PE, PAUKL, PUKLA, REKL, PDP, $\mathrm{err}_{\mathrm{KL}}$ and  $\mathrm{err}_{\ell^2}$ is reached.
We report in table \ref{tab:table iterations} the mean value and the standard deviation of this number calculated on the $25$ reconstructions.

\begin{table}[ht]
\centering

\resizebox{\linewidth}{!}{    \begin{tabular}{cccccccc}
\toprule
Statistic        &PE   &PAUKL   &PUKLA  &REKL   &PDP     &$\mathrm{err}_{\mathrm{KL}}$    &$\mathrm{err}_{\ell^2}$ \\
\midrule
\multicolumn{8}{c}{\textbf{NGC7027}}\\
\cmidrule{2-8}
Low    &$76_{(\pm 4)}$ &$78_{(\pm 4)}$ &$78_{(\pm 5)}$ &$79_{(\pm 6)}$ &$31_{(\pm 6)}$ &$89_{(\pm 3)}$ &$52_{(\pm 3)}$\\ 
Medium &$312_{(\pm 13)}$ &$303_{(\pm 24)}$ &$306_{(\pm 28)}$ &$303_{(\pm 29)}$ &$82_{(\pm 10)}$ &$402_{(\pm 14)}$ &$255_{(\pm 11)}$\\ 
High   &$2229_{(\pm 121)}$ &$2311_{(\pm 223)}$ &$2308_{(\pm 245)}$ &$2329_{(\pm 358)}$ &$373_{(\pm 41)}$ &$2836_{(\pm 122)}$ &$1612_{(\pm 92)}$\\ 

\midrule
\multicolumn{8}{c}{\textbf{Horse Head}}\\
\cmidrule{2-8}
Low    &$7_{(\pm 1)}$ &$7_{(\pm 1)}$ &$6_{(\pm 1)}$ &$7_{(\pm 0)}$ &$4_{(\pm 0)}$ &$5_{(\pm 0)}$ &$6_{(\pm 0)}$\\ 
Medium &$33_{(\pm 1)}$ &$32_{(\pm 2)}$ &$32_{(\pm 2)}$ &$32_{(\pm 2)}$ &$10_{(\pm 1)}$ &$33_{(\pm 1)}$ &$38_{(\pm 1)}$\\ 
High   &$185_{(\pm 11)}$ &$186_{(\pm 10)}$ &$188_{(\pm 11)}$ &$186_{(\pm 13)}$ &$44_{(\pm 3)}$ &$213_{(\pm 9)}$ &$245_{(\pm 9)}$\\ 

\bottomrule
\end{tabular}}

\caption{Comparison between the number of iterations that minimizes the PE, the estimators of the predictive risk (PAUKL, PUKLA and REKL), the PDP and the reconstruction errors ($\mathrm{err}_{\mathrm{KL}}$ and $\mathrm{err}_{\ell^2}$). We report the mean value and the standard deviation corresponding to $25$ Poisson noise realizations of the data.}

\label{tab:table iterations}
\end{table}

From the numerical tests it is evident that the performances of the three estimators of the risk, i.e. PAUKL, REKL and PUKLA, are very similar as every one of them is able to give an accurate estimate of the SPR without the knowledge of $\lambda$ (see figures \ref{fig:nebula risk medium stat}(a), \ref{fig:nebula risk medium stat}(b)). 
The advantage of the proposed estimator with respect to REKL and PUKLA is that it is not biased with a constant and unknown quantity, but it gives an absolute quantification of the predictive risk. 
This is shown in figure \ref{fig:nebula risk medium stat}(c) where one can notice that the SPR falls inside the standard deviation of PAUKL. 
Moreover, we can observe that the stopping rule based on the Poisson discrepancy principle is almost everywhere too much conservative with respect to the rules based on the predictive risk. 
This fact results in figures \ref{fig:ngc7027 low stat} and \ref{fig:horsehead high stat}, where the reconstructions obtained with the discrepancy principle appear too much regularized. 
It is also evident in table \ref{tab:table iterations}, where the mean value of the number of the iterations at which the Poisson discrepancy principle is satisfied is much lower than the mean values of the number of iterations that correspond to the minimum of the quantities $\mathrm{err}_{\mathrm{KL}}$ and $\mathrm{err}_{\ell^2}$. 
On the other hand, the discrepancy principle seems to be more stable as a stopping rule as the standard deviation of the number of iterations at which the PDP is minimized is much lower with respect to the standard deviation of the number of iterations at which the estimators of the risk reach the minimum (see table \ref{tab:table iterations}). 
Finally, despite the ill-posedness of the inverse problem, the predictive risk is an effective rule for stopping the EM algorithm.
The order of magnitude of the number of iterations at which $\mathrm{ER}_{\mathrm{KL}}$ or $\mathrm{ER}_{\ell^2}$ is minimized is the same of the number of iterations at which the SPR or the estimators of the risk are minimized (see figure \ref{fig:nebula risk medium stat}(d)). Same considerations hold true also for the PE and the reconstruction errors $\mathrm{err}_{\mathrm{KL}}$ and $\mathrm{err}_{\ell^2}$ in the case of a single realization of the noise (see table \ref{tab:table iterations}).
The EM algorithm with Poisson data is slow, therefore even if the iteration selected by a predictive risk stopping rule is not exactly the same as the one that corresponds to the minimum of $\mathrm{err}_{\mathrm{KL}}$ or $\mathrm{err}_{\ell^2}$, but it is rather close to it, then the difference in the reconstructed images is almost negligible.

\begin{figure}[t]
\centering
\subfloat[]{\includegraphics[width= .48\textwidth]{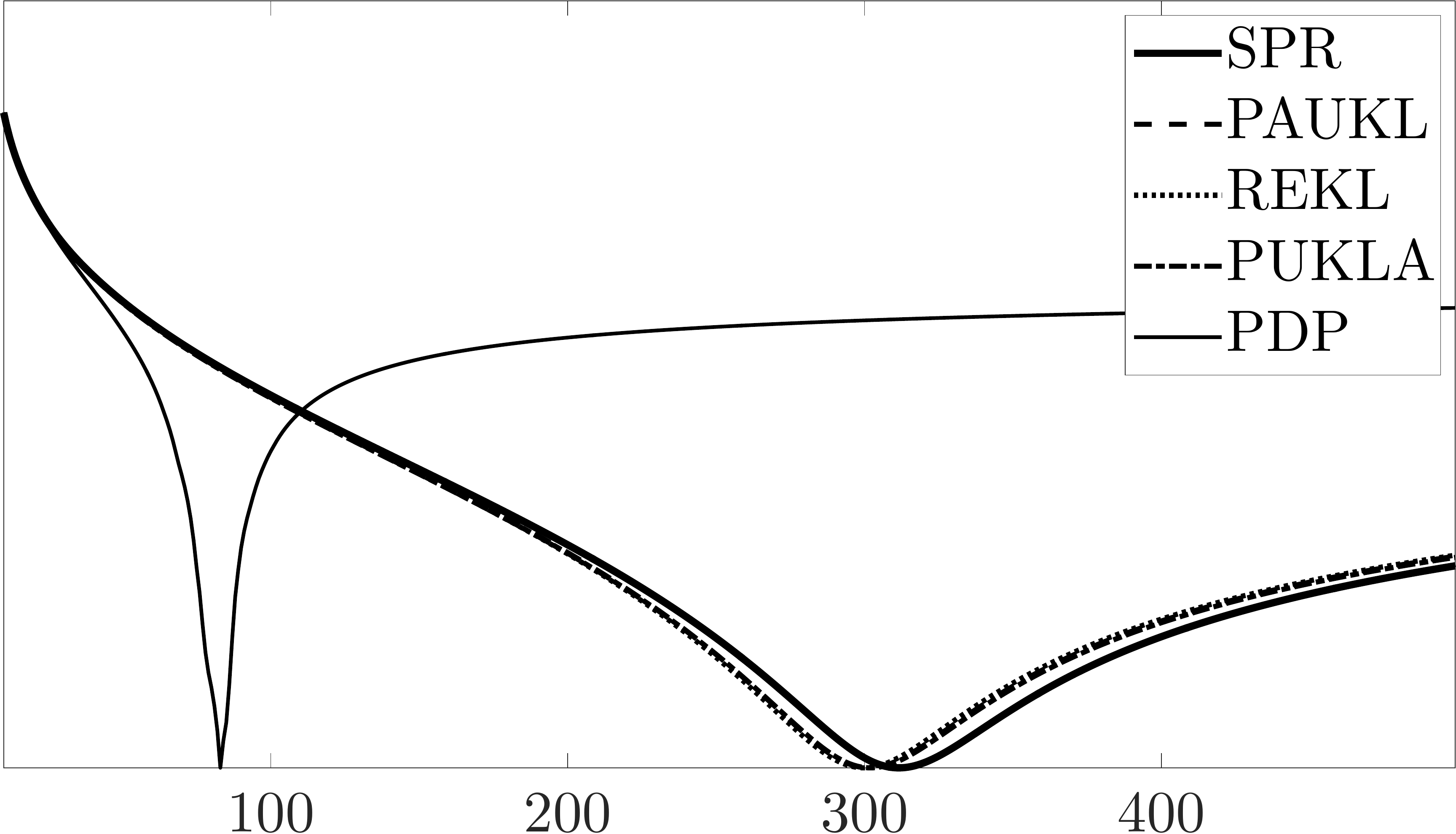}}
\hspace{0.1pt}
\subfloat[]{\includegraphics[width= .48\textwidth]{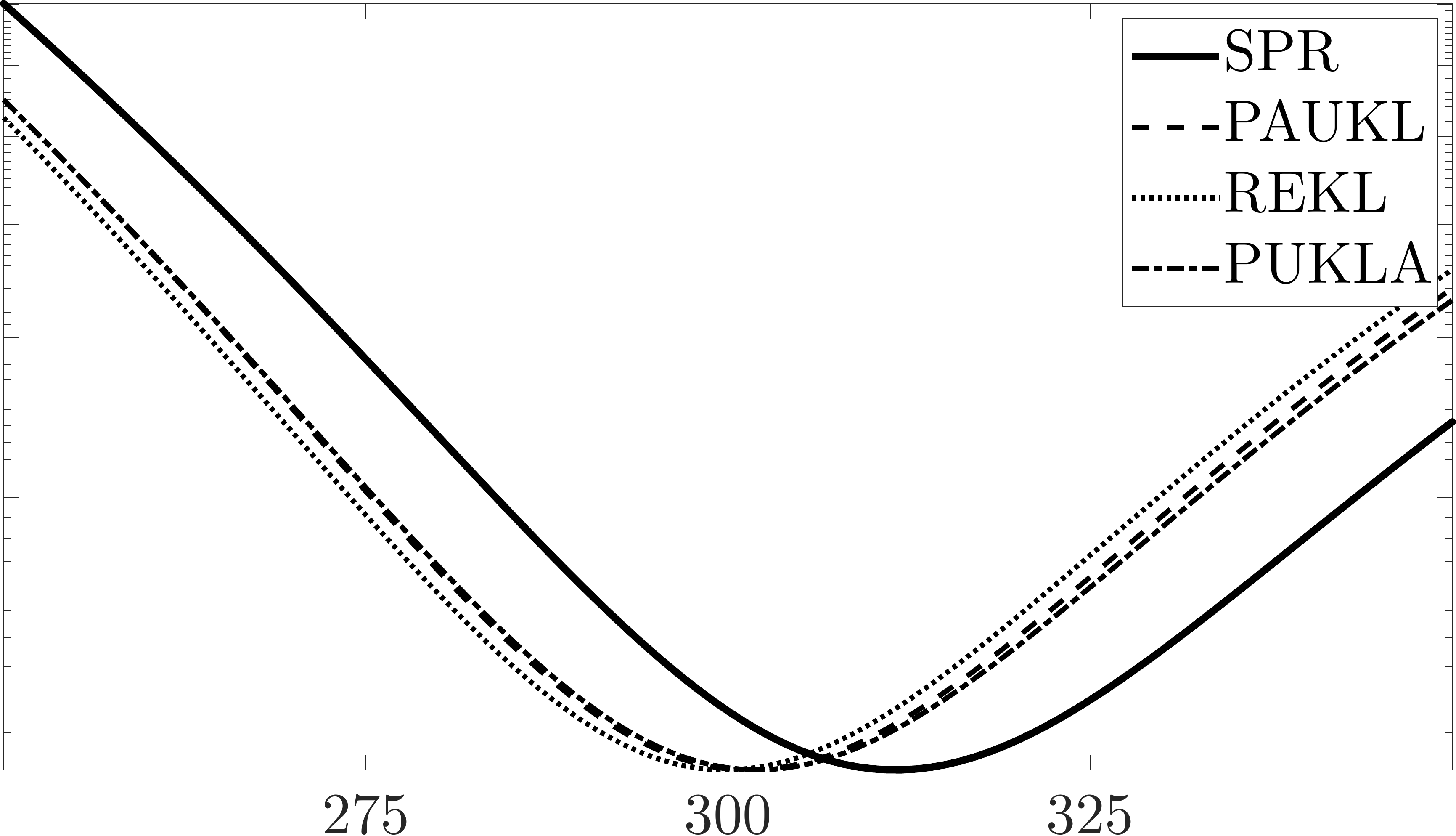}} \\

\subfloat[]{\includegraphics[width= .48\textwidth]{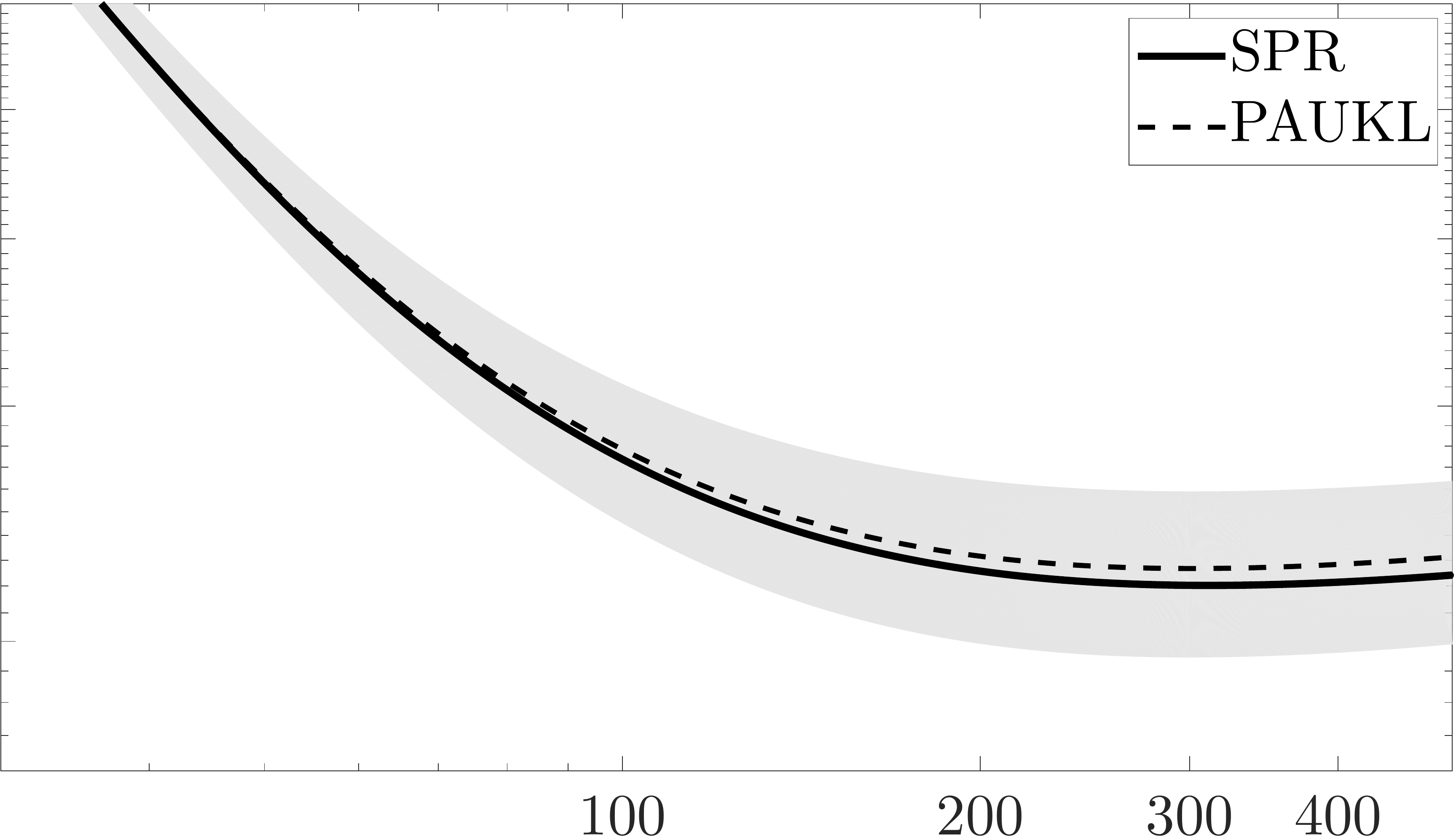}}
\hspace{0.1pt}
\subfloat[][]{\includegraphics[width= .48\textwidth]{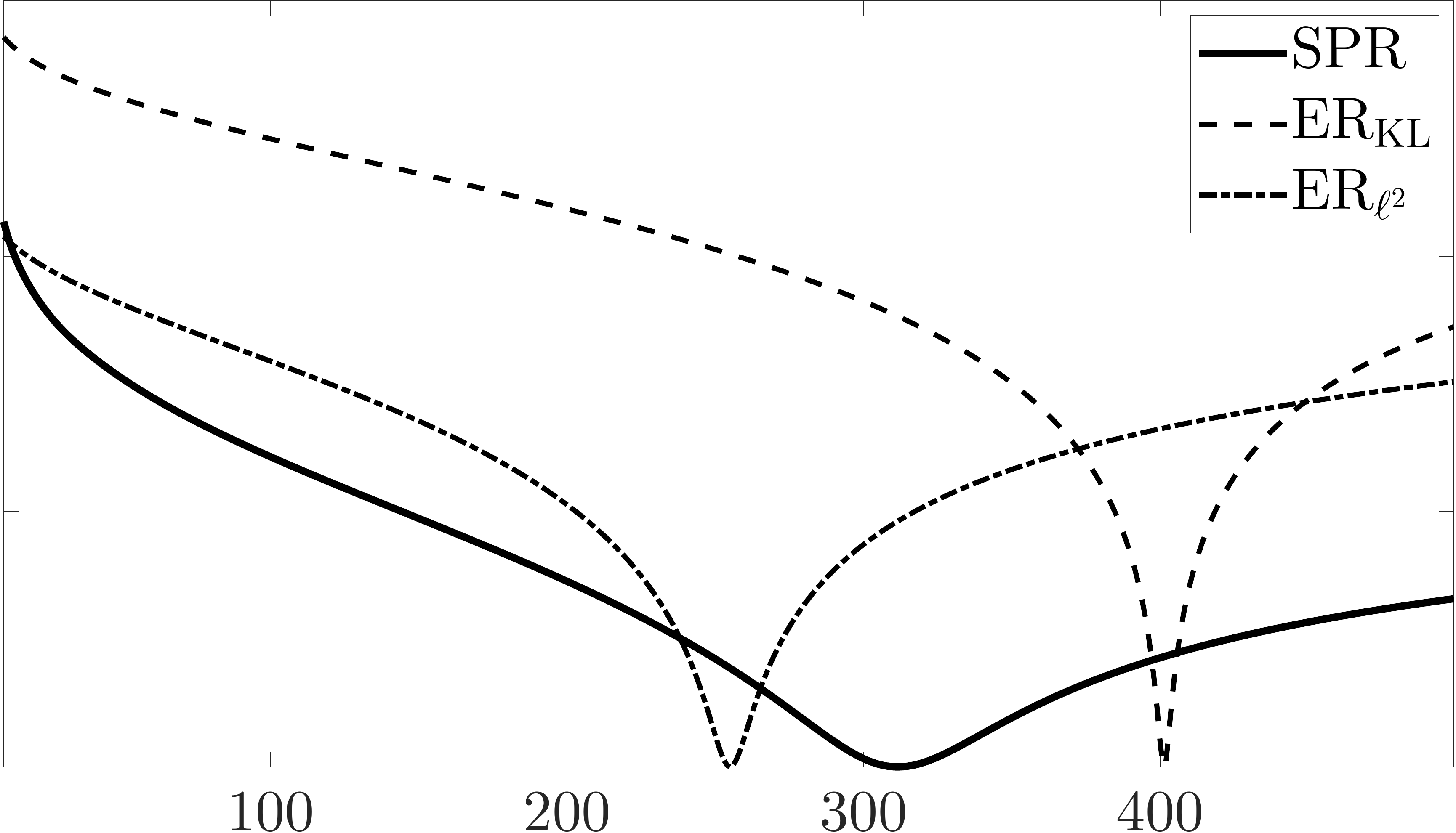}}

\caption{Results obtained on the image of nebula NGC7027 in the medium statistic case. Figure (a): comparison between the SPR and the mean values of PAUKL, REKL, PUKLA and PDP on $25$ Poisson noise realizations of the data.
Figure (b): zoom of figure (a). Figure (c): comparison between the SPR (solid line) and the mean value PAUKL (dashed line) on $25$ realizations of the data. The shaded area corresponds to the standard deviation of PAUKL. Figure (d): comparison between the SPR (solid line) and the empirical risks $\mathrm{ER}_{\mathrm{KL}}$ (dashed line) and $\mathrm{ER}_{\ell^2}$ (dot-dashed line). 
In each figure the number of iterations is reported on the x-axis. 
Figures (a)-(d) are plotted in logarithmic scale on the y-axis and figure (c) is plotted in logarithmic scale also on the x-axis.}
\label{fig:nebula risk medium stat}

\end{figure}

\subsection{Non-inverse crime setting}
In the inverse crime setting \cite{colton2019inverse} the operator used in the reconstruction process is exactly the same of the one used for generating the synthetic data.
Therefore, the tests we performed and showed in section \ref{inverse crime setting} are compliant with the hypotheses of the theorems and must reflect their theoretical properties.
However, in real world applications, the {\it exact} PSF of an acquisition instrument is not known (it does not necessarily exist) and therefore an {\it approximated} PSF is often used for the reconstruction process.
In this case the forward model used to generate the data is not exactly the same used during the inversion process and this framework is called non-inverse crime setting.
In order to study the behaviour of the stopping criteria presented in the previous section, in the non-inverse crime setting, inspired by the idea that a signal formation process is usually more complex than its mathematical formalization, first we simulate a more realistic image formation process by using an uneven PSF to generate the data and then we perform the inversion by using a smooth approximation of such a PSF.
Therefore, we consider to generate the exact PSF of a given image formation process by following the procedure proposed in \cite{benvenuto2010joint}: first we multiply by $10^4$ the PSF used in the tests of section \ref{inverse crime setting}, then we add Poisson noise and finally we normalize the result.
In this way we obtain an irregularly shaped PSF which we consider to be the {\it exact} PSF.
Then,  in the reconstruction process we make use of the smooth Gaussian shaped PSF described in section \ref{inverse crime setting} intended to be a smooth approximation of the exact PSF. 
We report in figure \ref{PSF comparison} a comparison between the exact and the approximated PSF.

\begin{figure}[ht]
\centering
\includegraphics[height=3.7cm]{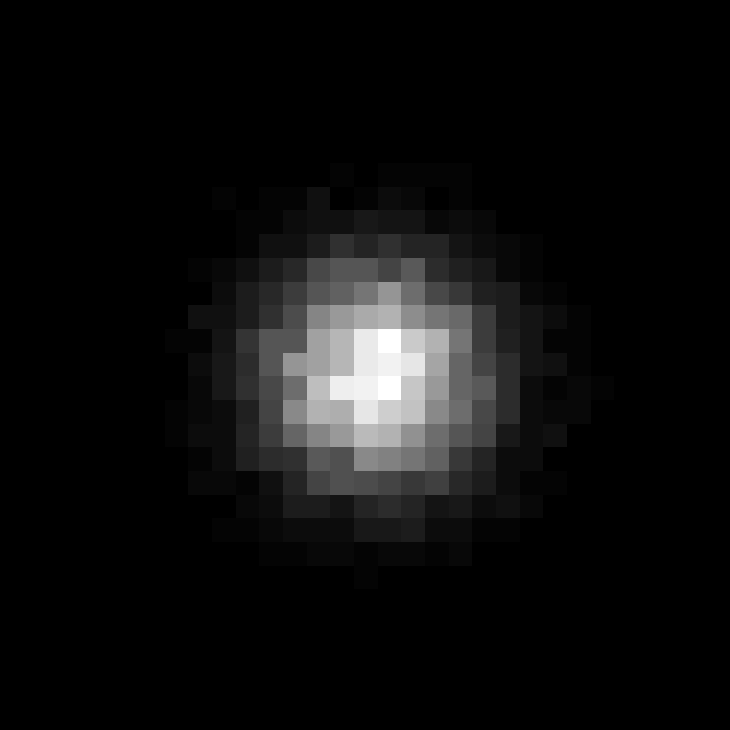}  
\includegraphics[height=3.7cm]{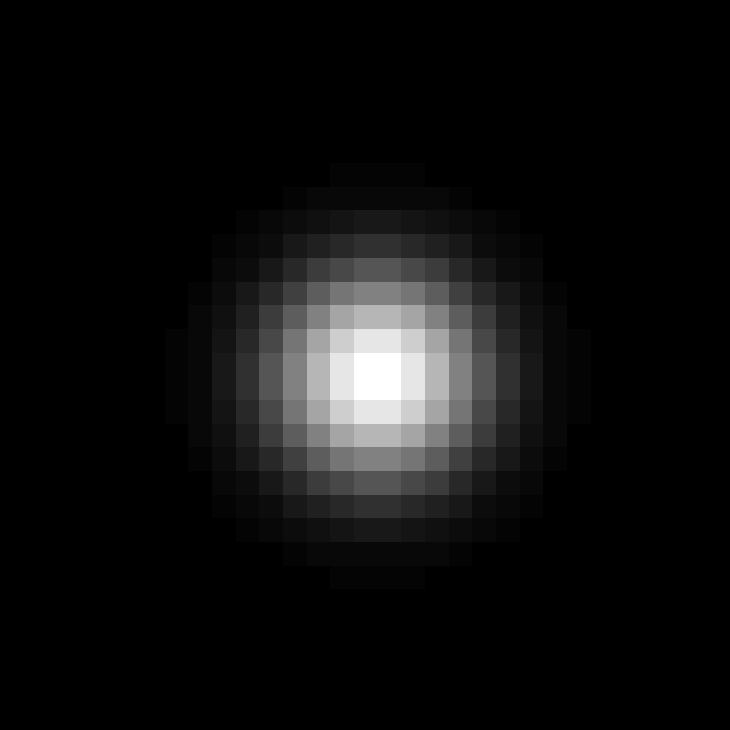}  
\includegraphics[height=3.7cm]{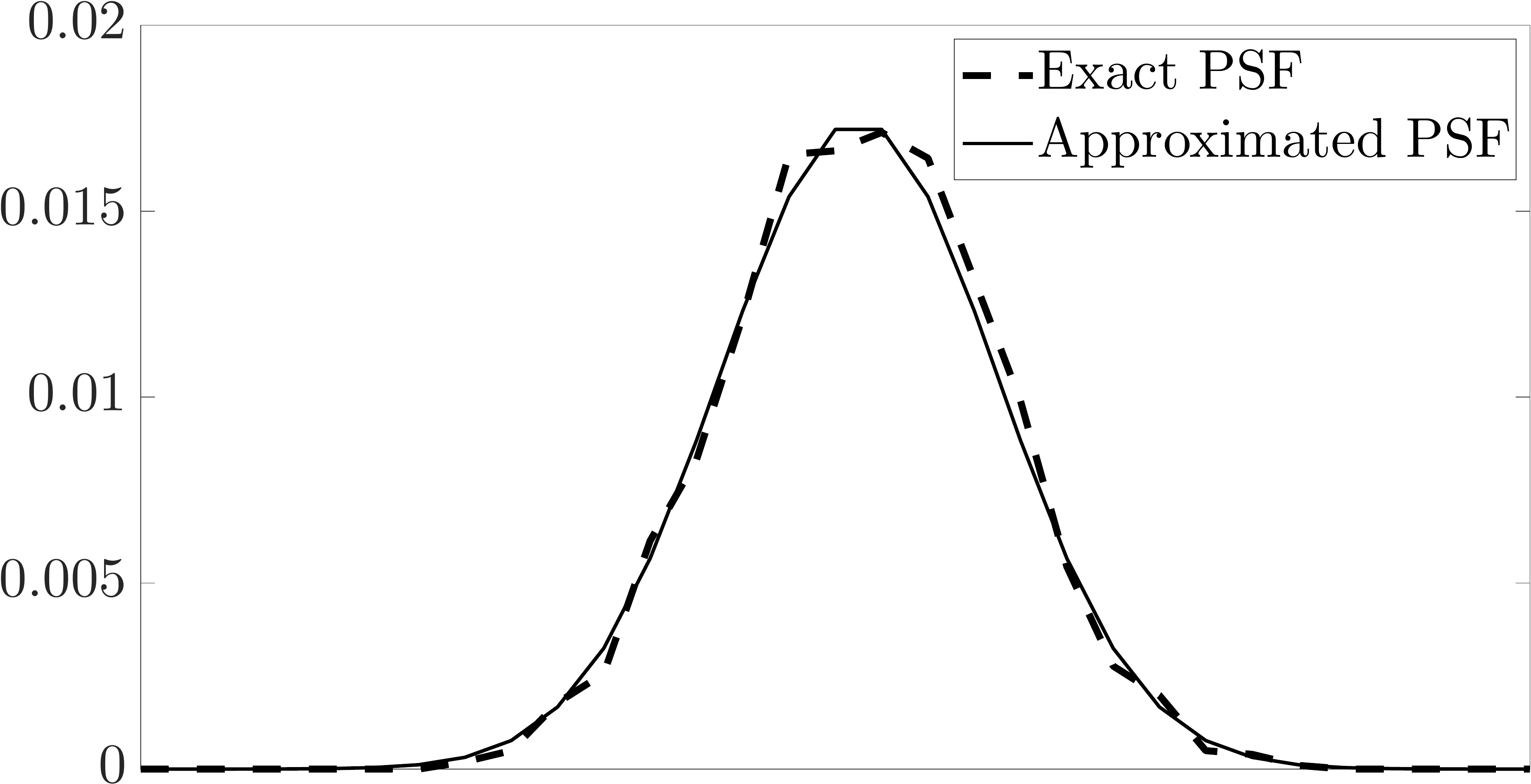} 
\caption{From left to right: the exact PSF, the approximated PSF and the profile of the exact (dashed line) and the approximated (solid line) PSF along an axis passing through the image center.}
\label{PSF comparison}
\end{figure}

It is well known that, in practical applications, the discrepancy principle in the Gaussian framework requires a tuning parameter to work \cite{engl1996regularization}. 
The main difficulty in real world problems is to determine a value for that parameters that makes the discrepancy principle sufficiently robust and reliable. 
We show now that, without a similar parameter, the discrepancy principle for Poisson data suffers from the same problem of the Gaussian case, i.e. sometimes it fails to work (as a stopping rule for the EM algorithm) in realistic simulation when the statistic of the data is high. 
Moreover, we show with some tests that the stopping rules based on the predictive risk seem to be more stable (at least in some cases and for some level of statistic). 
Finally, we give some intuitive explanations for these behaviours (under appropriate hypothesis), without the presumption of being exhaustive in the treatment, but with the only aim to highlight the necessity of a deeper investigation concerning errors in the forward model.

We perform reconstructions from data generated by applying the irregularly shaped PSF to the images of nebula NGC7027 and of the Horse Head nebula. 
Again, we consider three levels of statistic: the total flux is equal to $5 \times 10^8$, $10^9$ and $5 \times 10^9$ for the NGC7027 nebula and it is equal to $1 \times 10^{10}$, $2.5 \times 10^{10}$ and $5 \times 10^{10}$ for the Horse Head nebula in the low, medium and high statistic case respectively.
For both images, the same background of the tests of section \ref{inverse crime setting} is added.
This time we perform reconstructions for only one realization of Poisson noise in order to verify what can be the behaviour of the different stopping rules in a realistic application.
We compute at every iteration the quantities described in section \ref{inverse crime setting}, i.e. the PE, the estimators of the risk (PAUKL, PUKLA and REKL), the reconstruction errors $\mathrm{err}_{\ell^2}$ and $\mathrm{err}_{\mathrm{KL}}$, the discrepancy $d_{\mathrm{KL}}$ and the PDP. 
We report in the left panel of figure \ref{fig:non-ic d_kl and risk} a plot of the discrepancy $d_{\mathrm{KL}}$ for the three different levels of statistic together with a horizontal line at $M/2$ computed for the nebula NGC7027 image. 
In the right panel of figure \ref{fig:non-ic d_kl and risk} instead, we show the behaviour of the PE and of the proposed estimator in the medium statistic case for the nebula NGC7027. 
We also compare their optimal iterates with the ones of the reconstruction errors. 
Finally, in table \ref{tab:non-ic table iterations} we report the number of iterations at which the quantities described above are minimized for the images of both nebula NGC7027 and Horse Head nebula in the three levels of statistic.

\begin{figure}[ht]
\centering
\includegraphics[width=.48\textwidth]{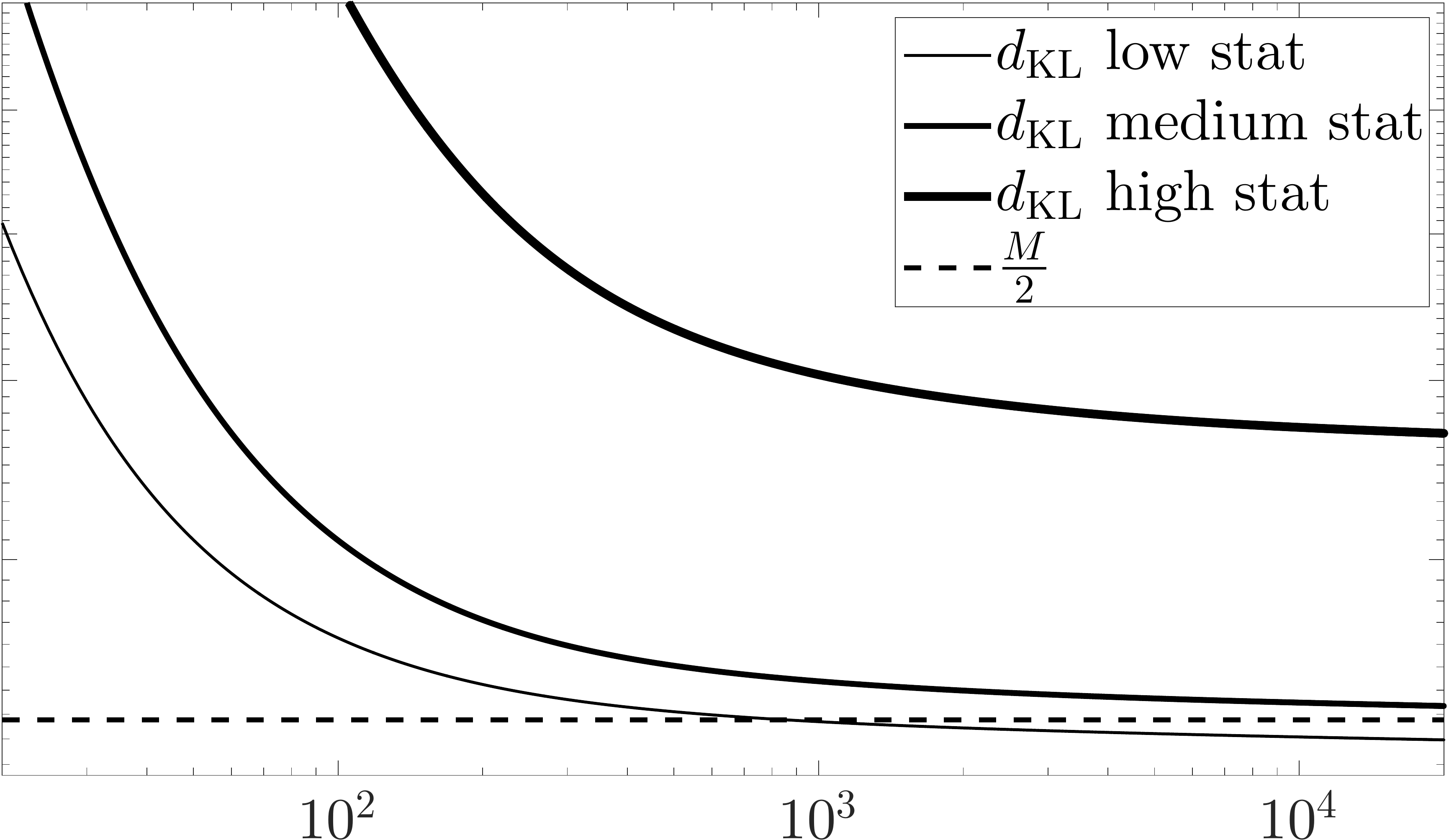}
\includegraphics[width=.48\textwidth]{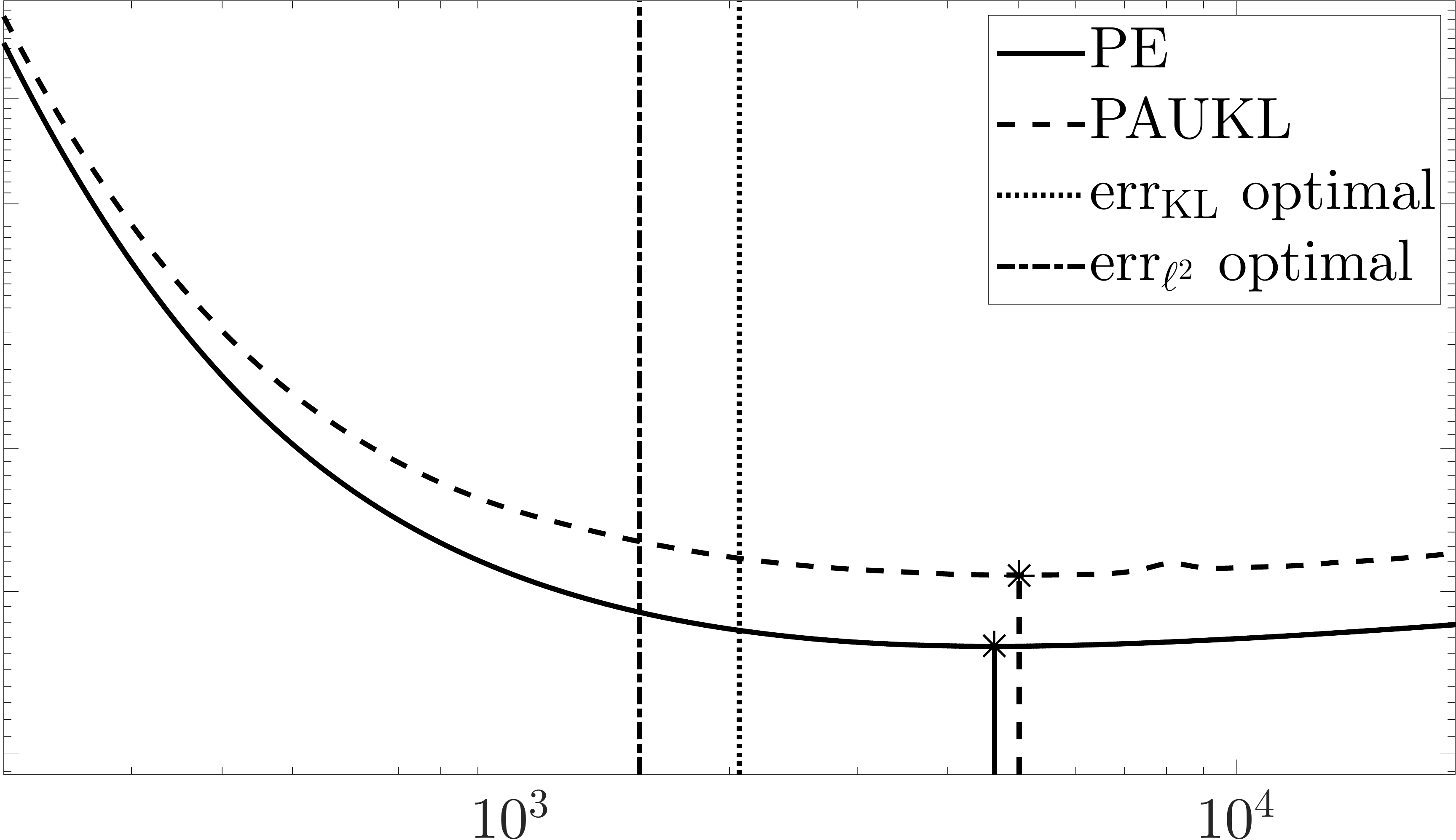}  

\caption{Left: discrepancy in KL divergence along the iterations for different levels of statistic for the nebula NGC7027. 
Solid lines represents the $d_{\mathrm{KL}}$ in the low, medium and high statistic case (low, medium and high thickness respectively) and dashed line is the asymptotic variance level.
Right: PE (solid line) and PAUKL (dashed line) computed for the nebula NGC7027 in the medium statistic case. We indicate with a vertical line the iterate at which $\mathrm{err}_{\mathrm{KL}}$ and $\mathrm{err}_{\mathrm{\ell^2}}$ are minimized (dot-dashed line and dotted line respectively). In each figure the number of iterations is reported on the x-axis. Both figures are plotted in logarithmic scale on the x and on the y-axis.}
\label{fig:non-ic d_kl and risk}
\end{figure}

\begin{table}[ht]
\centering
\begin{tabular}{cccccccc}
\toprule
Statistic        &PE   &PAUKL   &PUKLA  &REKL   &PDP     &$\mathrm{err}_{\mathrm{KL}}$    &$\mathrm{err}_{\ell^2}$ \\
\midrule
\multicolumn{8}{c}{\textbf{NGC7027}}\\
\cmidrule{2-8}
Low    &$1606$ &$1719$ &$1537$ &$1653$ &$857$ &$1224$ &$832$\\ 
Medium &$4639$ &$5014$ &$6012$ &$3985$ &$\ast$ &$2063$ &$1503$\\ 
High   &$\ast$ &$\ast$ &$\ast$ &$\ast$ &$\ast$ &$3026$ &$2869$\\ 
\midrule
\multicolumn{8}{c}{\textbf{Horse Head}}\\
\cmidrule{2-8}
Low    &$1874$ &$1749$ &$1817$ &$1718$ &$689$ &$1295$ &$1457$\\ 
Medium &$5647$ &$5530$ &$6246$ &$5072$ &$\ast$ &$2195$ &$2476$\\ 
High   &$\ast$ &$\ast$ &$\ast$ &$\ast$ &$\ast$ &$3056$ &$3468$\\  
\bottomrule
\end{tabular}
\caption{Comparison between the number of iterations that minimizes the PE, the estimators of the predictive risk (PAUKL, PUKLA and REKL), the PDP and the reconstruction errors ($\mathrm{err}_{\mathrm{KL}}$ and $\mathrm{err}_{\ell^2}$). We indicate with an asterisk if after $2\times 10^4$ iterations the minimum has not been reached.}
\label{tab:non-ic table iterations}
\end{table}

We observe that the Poisson discrepancy principle is satisfied just in the low statistic case for both images (see table \ref{tab:non-ic table iterations}). 
An explanation for this behaviour can be found in figure \ref{fig:non-ic d_kl and risk} left panel.
The line corresponding to the discrepancy gets more and more high with increasing statistic while the asymptotic variance does not depend on the statistic of the data. 
The error introduced by approximating an uneven PSF with an ideal one is reflected in an offset in the discrepancy that does not let it go below $M/2$ (neither for $k=+\infty$). 
In the next Lemma we show that such an offset depends on the specific data sample and on its geometric relation with respect to the image of the nonnegative orthant through the forward operator.
Throughout this work we denoted by
$H$ the matrix corresponding to the ideal PSF. 
Here we denote by
$\hat H$ the matrix corresponding to the exact PSF.
Given a matrix $C\in\mathbb{R}^{M \times N}$ we also define the cone
\begin{equation}
\mathcal{R}_+(C) \coloneqq \left\{Cx \,\colon\, x\in\mathbb{R}^N_+ \right\} ~,
\end{equation}
which contains all the possible values of a nonnegative vector through the matrix operator $C$. 
Hereafter, to simplify the computation and give a direct intuition of the next Lemma, we consider the EM algorithm with background $b=0$. The function $R_k$ defined in \eref{reg_alg} is then positively homogeneous of degree $1$, i.e. $R_k(Ly) = LR_k(y)$ for all $L>0$ and for all $y\in\mathbb{R}^M_+$ (see \cite{benvenuto2017study}).

\begin{lemma}\label{Lemma: discrpancy failure}
Let $y$ be a realization of $Y\sim \mathcal{P}(\hat Hx)$ such that $y\notin \mathcal{R}_+(H)$.
Then, if $\Vert y \Vert$ is sufficiently large,
for all $k\in\mathbb{N}$
\begin{equation}
\label{PDP}
D_{\mathrm{KL}}(y, HR_k(y)) > \frac{M}{2} ~ .
\end{equation}
\end{lemma}
\begin{proof}
From the positive homogeneity of $R_k$ it follows
\begin{equation}
\label{DKL_positive}
D_{\mathrm{KL}}(y, HR_k(y)) = \| y \| D_{\mathrm{KL}}\left(\bar{y}, HR_k\left(\bar{y}\right)\right) ~,
\end{equation}
where $\bar{y} \coloneqq y/\Vert y \Vert$. 
Now, thanks to the properties of the EM algorithm (see \cite{dempster1977maximum, shepp1982maximum}), for all $k\in\mathbb{N}$ we have
\begin{equation}
\fl D_{\mathrm{KL}}\left(\bar{y}, HR_k\left(\bar{y}\right)\right) \geq  \lim_{k\to +\infty} D_{\mathrm{KL}}\left(\bar{y}, HR_k\left(\bar{y}\right)\right) = D_{\mathrm{KL}}(\bar{y}, \lim_{k\to +\infty} HR_k\left(\bar{y}\right))  ~.
\end{equation}
As the sequence $\{ HR_k(\bar{y}) \}_{k\in\mathbb{N}}$ converges to $y^\ast \in \mathcal{R}_+(H)$ and as $\bar{y}\notin\mathcal{R}_+(H)$, then
\begin{equation}
D_{\mathrm{KL}}\left(\bar{y}, HR_k\left(\bar{y}\right)\right) \geq D_{\mathrm{KL}}(\bar{y},  y^\ast) >0 ~.
\end{equation}
Then, by taking 
\begin{equation}
\label{PDP-fail-condition}
\| y \|\geq M/2 (D_{\mathrm{KL}}(\bar{y},  y^\ast))^{-1} ~ 
\end{equation}
the thesis is straightforward.
\end{proof}

We point out that the conditions under which we prove Lemma \ref{Lemma: discrpancy failure} do not meet entirely the conditions under which we have developed the theory of this paper: this result is meant just to be a starting point for further investigations. 
In view of this, the thesis of Lemma \ref{Lemma: discrpancy failure} holds true even if $y$ is a realization of $Y \sim \mathcal{P}(Hx)$, i.e. in the inverse crime case, as the crucial hypothesis is that $y\notin\mathcal{R}_+(H)$.
However, for large scale problem the realization $y$ is typically not too far from $\mathcal{R}_+(H)$ in the sense that $D_{\mathrm{KL}}(\bar y,y^*)$ is small.
In this scenario, the norm $\|y\|$ must be so large to satisfy equation \eref{PDP-fail-condition} that this case is really unlikely to happen in practice.
Furthermore, an extremely large number of counts corresponds to a large signal-to-noise ratio and then regularization plays a minor role.
What does really matter in this setting is the systematic error in the forward modelling which may increase $D_{\mathrm{KL}}(\bar y,y^*)$ to such an extend that the Poisson discrepancy principle does not work.
An example of this is shown in the left panel of figure \ref{fig:non-ic d_kl and risk}.

On the other hand, the stopping rule based on our estimator selects the iterate $k$ that satisfies 
\begin{equation}
D_{\mathrm{KL}}(y, HR_k(y)) - D_{\mathrm{KL}}(y, HR_{k+1}(y))  < (y\nabla) \cdot \log \frac{HR_{k+1}( y )}
{HR_k( y )}
\end{equation}
and therefore does not suffer from the presence of an offset in the case in which $y\notin \mathcal{R}_+(H)$; this fact makes risk-based methods usually more robust with respect to the Poisson discrepancy principle. 
The stopping rule is always effective as long as the rate with which the KL divergence decreases along the iterations is definitely lower than the rate with which the term involving the trace increases. 
However, in the non-inverse crime setting there are problems due to the incorrect forward model used that make our stopping rule not work with high statistic level (see table \ref{tab:non-ic table iterations}). 
It is difficult to formalize the reasons of this misbehaviour as we did in Lemma \ref{Lemma: discrpancy failure} for the Poisson discrepancy principle. However, the crucial point is in the ill-posedness of the inverse problem since we notice that as long as the predictive risk estimator does not reach a minimum value along the iterations, also the PE, which is computed with the knowledge of $\lambda$, does not reach a minimum (see table \ref{tab:non-ic table iterations}, high statistic case).

\section{Conclusions}

As already remarked, in the case of the EM algorithm, the above mentioned predictive risk estimators do not differ each other, providing an alternative to the Poisson discrepancy principle, exactly as SURE does with respect to the Morozov's discrepancy principle. 
However, the proven theoretical results provide a more sound basis of the proposed estimator and allow identifying the analogy between the Gaussian and the Poisson case. 
Besides proving a more general Stein's Lemma for Poisson data, we also introduced an estimator that is asymptotically unbiased, approximating the exact value of the predictive risk, unlike PUKLA or REKL that include an unknown constant offset as bias term, while preserving the same computational efficiency.
Moreover, we give some simple conditions which has to be verified by a regularization method in order to make this estimator asymptotically unbiased.
Therefore, this estimator can be efficiently used in selecting the regularization parameter of any algorithm that satisfies such conditions.
The need of a suitable regularization is evident when the object to retrieve contains structures like edges or finer details. 
The analysis of other regularization methods may be the next direction of this work.
Finally, as suggested by the numerical experiments reported in this paper, a theoretical investigation of the non-inverse crime case, also presented in this paper, is a crucial point for understanding the range of applicability of the proposed method.

\ack
The authors thank the National Group of Scientific Computing (GNCS-INDAM) that supported this research.

\appendix
\section*{Appendix}
\setcounter{section}{1}
In the following we will denote $\pi_\lambda$ the probability mass function of an $M $-dimensional Poisson random variable with independent components and mean value $\lambda\in\mathbb{R}^M_{++}$, i.e.
\begin{equation}
\pi_\lambda\left(n\right) \coloneqq \prod_{i=1}^M \frac{e^{-\lambda_i} \lambda_i^{n_i}}{n_i !} \quad\forall{n\in\mathbb{N}^M}~.
\end{equation}
Moreover, if $I\subseteq \{1, \dots, M \}$ is a subset of indices, we will pose
\begin{equation}
\Vert y \Vert_I \coloneqq \sqrt{\sum_{i \in I} y_i^2}
\end{equation}
for all $y\in\mathbb{R}^M$.
Finally, the inequalities between vectors and the integer part $\lfloor \cdot \rfloor$ applied to vectors are meant component-wise.

The result proved in the next Lemma follows what done in \cite{zanella2013corrigendum}.
\begin{lemma}\label{Lemma: exp}
There exist $\eta>0$ and $0 < \theta < e$ such that, for all $x > 0$, we have
\begin{equation}
g(x)\coloneqq \frac{x^{\left\lfloor \frac{x}{2} \right\rfloor}}{\left\lfloor \frac{x}{2} \right\rfloor !} \leq \eta \theta^{\frac{x}{2}} e^{\frac{x}{2}}
\end{equation}
\end{lemma}
\begin{proof}
We consider three cases.
\begin{enumerate}[label= \arabic*)]

\item If $0< x \leq 2$, then $g(x)\leq 2$.

\item If $2< x \leq 6$, by using the Stirling lower bound of the factorial, we obtain
\begin{equation}
\left\lfloor \frac{x}{2} \right\rfloor ! \geq \sqrt{2 \pi} \left\lfloor \frac{x}{2} \right\rfloor^{\left\lfloor \frac{x}{2} \right\rfloor + \frac{1}{2}} e^{-\left\lfloor \frac{x}{2} \right\rfloor} \geq \sqrt{2 \pi} \left\lfloor \frac{x}{2} \right\rfloor^{\left\lfloor \frac{x}{2} \right\rfloor} e^{-\frac{x}{2}} ~,
\end{equation}
therefore
\begin{equation}\label{bound g funct}
g(x) \leq \frac{1}{\sqrt{2\pi}} e^{\frac{x}{2}} \left( \frac{x}{\left\lfloor \frac{x}{2} \right\rfloor} \right)^{\left\lfloor \frac{x}{2} \right\rfloor}  ~.
\end{equation}
Moreover, as
\begin{equation}
\left( \frac{x}{\left\lfloor \frac{x}{2} \right\rfloor} \right)^{\left\lfloor \frac{x}{2} \right\rfloor} = e^{\left\lfloor \frac{x}{2} \right\rfloor \log x } e^{-\left\lfloor \frac{x}{2} \right\rfloor \log\left\lfloor \frac{x}{2} \right\rfloor} \leq e^{3 \log 6} = 216 ~,
\end{equation}
we obtain
\begin{equation}
g(x) \leq \frac{216}{\sqrt{2 \pi}} e^{\frac{x}{2}} ~.
\end{equation}

\item If $x > 6$, using again \eref{bound g funct} and observing that
\begin{equation}
2 \leq\frac{x}{\left\lfloor \frac{x}{2} \right\rfloor} \leq \frac{8}{3} \simeq 2.67 < e
\end{equation}
leads to 
\begin{equation}
g(x) \leq \frac{1}{\sqrt{2\pi}} e^{\frac{x}{2}} \left( \frac{8}{3}\right)^{\frac{x}{2}} ~.
\end{equation} 
\end{enumerate}
We conclude by posing $\eta = \frac{216}{\sqrt{2\pi}}$ and $\theta = \frac{8}{3}$.
\end{proof}

\begin{lemma}\label{Lemma: O grande exp}
There exist $c>0$ such that
\begin{equation}
\sum_{n \leq \left\lfloor \frac{\lambda}{2}\right\rfloor} \pi_\lambda(n) = O \left( \Vert \lambda \Vert^M e^{- c \Vert \lambda \Vert} \right) 
\end{equation}
as $\Vert \lambda \Vert \to +\infty$ in $\mathbb{R}^M_{++}$.
\end{lemma}
\begin{proof}
By definition we have
\begin{equation}
\sum_{n \leq \left\lfloor \frac{\lambda}{2}\right\rfloor} \pi_\lambda(n) = \prod_{i=1}^M  \sum_{n_i = 1}^{\left\lfloor \frac{\lambda_i}{2}\right\rfloor} \frac{e^{-\lambda_i} \lambda_i^{n_i}}{n_i !}  ~.
\end{equation}
Following what done in \cite{zanella2013corrigendum}, by using the fact that $\lambda_i^r/r! < \lambda_i^s/s!$ when $r < s \leq \left\lfloor \frac{\lambda_i}{2} \right\rfloor$, we obtain
\begin{equation}
\sum_{n_i = 0}^{\left\lfloor \frac{\lambda_i}{2} \right\rfloor} \frac{e^{-\lambda_i} \lambda_i^{n_i}}{{n_i}!} \leq \left( \left\lfloor \frac{\lambda_i}{2} \right\rfloor + 1 \right) e^{-\lambda_i} \frac{\lambda_i^{\left\lfloor \frac{\lambda_i}{2} \right\rfloor}}{\left\lfloor \frac{\lambda_i}{2} \right\rfloor !} ~.
\end{equation}
Now, thanks to Lemma \ref{Lemma: exp}, there exist $\eta>0$ and $0 < \theta <e$ such that
\begin{equation}
\frac{\lambda_i^{\left\lfloor \frac{\lambda_i}{2} \right\rfloor}}{\left\lfloor \frac{\lambda_i}{2} \right\rfloor !} \leq \eta \theta^{\frac{\lambda_i}{2}} e^{\frac{\lambda_i}{2}} ~,
\end{equation}
therefore
\begin{equation}
\sum_{n_i=0}^{\left\lfloor \frac{\lambda_i}{2} \right\rfloor} \frac{e^{-\lambda_i} \lambda_i^{n_i}}{n_i!} \leq \eta \left( \left\lfloor \frac{\lambda_i}{2} \right\rfloor + 1 \right) e^{-\lambda_i} \theta^{\frac{\lambda_i}{2}} e^{\frac{\lambda_i}{2}} \leq \eta \left(\frac{\lambda_i}{2} + 1 \right)\left(\frac{e}{\theta}\right)^{-\frac{\lambda_i}{2}} ~. 
\end{equation}
If we denote $a \coloneqq e/\theta > 1$ and if we use the fact that $\lambda_i \leq \Vert \lambda \Vert$ in $\mathbb{R}^M_{++}$, we have
\begin{equation}
\prod_{i=1}^M  \sum_{n_i = 1}^{\left\lfloor \frac{\lambda_i}{2}\right\rfloor} \frac{e^{-\lambda_i} \lambda_i^{n_i}}{n_i !}  \leq \eta^M \left( \frac{\Vert \lambda \Vert}{2}  + 1 \right)^M a^{-\frac{\Vert \lambda \Vert_1}{2}} ~.
\end{equation}
Hence, as $\Vert \lambda \Vert_1 \geq \Vert \lambda \Vert$, we obtain
\begin{equation}
\fl \sum_{n \leq \left\lfloor \frac{\lambda}{2}\right\rfloor} \pi_\lambda(n) \leq \eta^M \left( \frac{\Vert \lambda \Vert}{2}  + 1 \right)^M a^{-\frac{\Vert \lambda \Vert}{2}} = \eta^M \left( \frac{\Vert \lambda \Vert}{2}  + 1 \right)^M e^{-\frac{\Vert \lambda \Vert}{2} \log a} ~.
\end{equation}
The thesis is straightforward if we pose $c \coloneqq (\log a)/2 >0$.
\end{proof}

\begin{lemma}\label{lemma equiv norme}
Let $I\subseteq \{1, \dots, M\}$ be a nonempty subset of indices and let $\mathcal{C}\subset\mathbb{R}^M_{++}$ be a closed convex cone. Then, there exist $c_1$, $c_2>0$ such that $c_1 \Vert y \Vert \leq \Vert y \Vert_I \leq c_2 \Vert y \Vert$ for all $y\in\mathcal{C}$.
\end{lemma}
\begin{proof}
We can trivially take $c_2 = 1$ as $\Vert y \Vert_I \leq \Vert y \Vert$. Now, let us pose
\begin{equation}
g(y) \coloneqq \frac{\Vert y \Vert_I}{\Vert y \Vert} ~.
\end{equation}
We want to prove that there exists $c_1>0$ such that $g(y)\geq c_1$ for all $y\in\mathcal{C}\setminus\{0\}$. We observe that $g$ is continuous on $\mathbb{R}^M_+\setminus \{ 0 \}$ and that it is positively homogeneous of degree 0, i.e. $g(Ly) = g(y)$ for all $y\in\mathbb{R}_+^M\setminus\{0\}$ and for all $L>0$. We point out that $\mathcal{C}\cap \mathbb{S}^{M-1}$ (where $\mathbb{S}^{M-1}$ denotes the surface of the unit sphere in $\mathbb{R}^M$) is a compact set, therefore, for Weierstrass Theorem, $g$ attains a minimum value $c_1$ on it. Moreover, as $\mathcal{C}\cap \mathbb{S}^{M-1}$ it holds $c_1>0$. If we consider $y\in\mathcal{C}\setminus \{0\}$, then, thanks to the positive homogeneity of degree 0 of $g$ and thanks to the fact that $y/\Vert y\Vert \in \mathcal{C}\cap \mathbb{S}^{M-1}$, we have
\begin{equation}
g(y) = g \left( \Vert y\Vert \frac{y}{\Vert y\Vert} \right) = g \left( \frac{y}{\Vert y\Vert} \right)\geq c_1   ~. 
\end{equation}

\end{proof}

\begin{proof}[Proof of Lemma \ref{Stein's lemma (Poisson)}]
By adding and subtracting $\lambda_i \partial_i f(\lambda)$, we have
\begin{eqnarray}
\mathbb{E}\left( \left( Y_i - \lambda_i \right) f\left( Y \right) - Y_i \partial_i f(Y) \right) = & \mathbb{E} \left( \left( Y_i - \lambda_i \right) f\left( Y \right) - \lambda_i \partial_i f(\lambda) \right) + \\ & \mathbb{E}\left( \lambda_i \partial_i f(\lambda) - Y_i \partial_i f(Y) \right)~. \nonumber
\end{eqnarray}
Hence, we prove
\begin{equation}\label{eq O grande 1}
\mathbb{E} \left( \left( Y_i - \lambda_i \right) f\left( Y \right) - \lambda_i \partial_i f(\lambda) \right) = O \left(\Vert \lambda \Vert^{-1/2} \right) 
\end{equation}
and
\begin{equation}\label{eq O grande 2}
\mathbb{E}\left( \lambda_i \partial_i f(\lambda) - Y_i \partial_i f(Y) \right) = O \left(\Vert \lambda \Vert^{-1/2} \right)
\end{equation}
as $\Vert \lambda \Vert\to +\infty$ in $\mathcal{C}$  (actually, we will give only the proof of \eref{eq O grande 1} as the one of \eref{eq O grande 2} is analogous). If we apply Taylor formula with Lagrange remainder to $f$ we have 
\begin{equation}
\fl f(Y) = f(\lambda) + \sum_{j=1}^M \partial_j f(\lambda)(Y_j - \lambda_j) + \frac{1}{2} \sum_{k, l=1}^M  \partial_k\partial_l f(\xi_{\lambda, Y})(Y_k - \lambda_k)(Y_l - \lambda_l) ~,
\end{equation}
where $\xi_{\lambda, Y} = \lambda + t(Y - \lambda)$ for some $t\in (0,1)$. Therefore,
\begin{equation}
\eqalign{&\mathbb{E} \left(\left( Y_i - \lambda_i \right) f\left( Y \right) - \lambda_i \partial_i f(\lambda)\right) = 
\\ & f(\lambda)\mathbb{E} \left( Y_i - \lambda_i \right) + \left[\sum_{j=1}^M \partial_j f(\lambda) \mathbb{E} \left(\left( Y_i - \lambda_i \right)\left( Y_j - \lambda_j \right)\right) - \lambda_i \partial_i f(\lambda)\right] + \\
&\frac{1}{2} \sum_{k,l=1}^M \mathbb{E} \left(\partial_k\partial_l f(\xi_{\lambda, Y}) \left( Y_i - \lambda_i \right) \left(Y_k - \lambda_k\right) \left(Y_l - \lambda_l\right) \right) ~.}
\end{equation}
Now, as $\mathbb{E} \left( Y_i - \lambda_i \right) = 0$ and $\mathbb{E} \left(\left( Y_i - \lambda_i \right)\left( Y_j - \lambda_j \right)\right) = \delta_{ij} \lambda_i$, where $\delta_{ij}$ denotes the Kronecker symbol, we obtain
\begin{equation}
\eqalign{&\mathbb{E} \left(\left( Y_i - \lambda_i \right) f\left( Y \right) - \lambda_i \partial_i f(\lambda)\right) = \cr & \frac{1}{2} \sum_{k,l=1}^M \mathbb{E} \left(\partial_k\partial_l f(\xi_{\lambda, Y}) \left( Y_i - \lambda_i \right) \left(Y_k - \lambda_k\right) \left(Y_l - \lambda_l\right) \right) .}
\end{equation}
Let us fix $k$, $l\in\{1, \dots, M \}$. We want to prove that
\begin{equation}\label{Lemma: tesi 1}
\mathbb{E} \left(\partial_k\partial_l f(\xi_{\lambda, Y}) \left( Y_i - \lambda_i \right) \left(Y_k - \lambda_k\right) \left(Y_l - \lambda_l\right) \right)^2 = O (\Vert \lambda \Vert^{-1})
\end{equation}
as $\Vert \lambda \Vert \to +\infty$ in $\mathcal{C}$. By applying Jensen's inequality and hypothesis \ref{Stein's Lemma hp iv} we obtain
\begin{equation}
\eqalign{&\left(\mathbb{E} \left(\partial_k\partial_l f(\xi_{\lambda, Y}) \left( Y_i - \lambda_i \right) \left(Y_k - \lambda_k\right) \left(Y_l - \lambda_l\right) \right) \right)^2 \leq \\
&\mathbb{E} \left(\left(\partial_k\partial_l f(\xi_{\lambda, Y})\right)^2 \left( Y_i - \lambda_i \right)^2 \left(Y_k - \lambda_k\right)^2 \left(Y_l - \lambda_l\right)^2 \right) \leq \\
&\mathbb{E} \left(\frac{c_{k,l}^2}{(\Vert \xi_{\lambda, Y} \Vert^2 + \varepsilon_{k,l})^2 } \left( Y_i - \lambda_i \right)^2 \left(Y_k - \lambda_k\right)^2 \left(Y_l - \lambda_l\right)^2 \right) = \\
&c_{k,l}^2 \sum_{n \in \mathbb{N}^M} \frac{1}{\left(\Vert \xi_{\lambda, n} \Vert^2 + \varepsilon_{k,l} \right)^2 } \left( n_i - \lambda_i \right)^2 \left(n_k - \lambda_k\right)^2 \left(n_l - \lambda_l\right)^2 \pi_\lambda( n) ~.} 
\end{equation}
We point out that $\mathbb{N}^M$ can be written as a finite union of disjoint sets as follows
\begin{equation}
\mathbb{N}^M = \bigcup_{p=0}^{2^M-1}  \mathcal{A}_\lambda^p ~,
\end{equation}
where, if we denote by $p_i$ the $i$-th bit in the binary representation of $p$, then
\begin{equation}
\fl\mathcal A_\lambda^p \coloneqq \left\{ n \in \mathbb N^M ~\colon~ (-1)^{p_i} \left(n_i - \left\lfloor \frac{\lambda_i}{2} \right\rfloor - p_i\right) \leq 0\, ~ \forall i\in\{1, \dots, M\} \right\} ~;
\end{equation}
moreover, we pose
\begin{equation}
\fl I^p_{\leq} \coloneqq \left\{ i\in\{1, \dots, M\}~\colon~ p_i = 0 \right\}~, \quad I^p_{\geq} \coloneqq \left\{ i\in\{1, \dots, M\}~\colon~ p_i = 1 \right\} ~.
\end{equation}
Now, it is easy to verify that 
\begin{equation}
\Vert \xi_{\lambda, n} \Vert \geq  \frac{\Vert \lambda \Vert_{I^p_{\geq}}}{2} ~,
\end{equation}
therefore
\begin{equation}\label{spezzamento A^p_lambda}
\fl\sum_{n \in \mathbb{N}^M} \frac{1}{\left(\Vert \xi_{\lambda, n} \Vert^2 + \varepsilon_{k,l} \right)^2 } \left( n_i - \lambda_i \right)^2 \left(n_k - \lambda_k\right)^2 \left(n_l - \lambda_l\right)^2 \pi_\lambda( n) \leq  \sum_{p=0}^{2^M - 1} \varphi_{p,k,l}\left( \lambda \right) ~,
\end{equation}
with
\begin{equation}
\fl\varphi_{p,k,l}\left(\lambda \right) \coloneqq \frac{16}{\left(\Vert \lambda \Vert_{I_\geq^p}^2 + 4 \varepsilon_{k,l} \right)^2} \sum_{n \in \mathcal{A}_\lambda^p} \left( n_i - \lambda_i \right)^2 \left(n_k - \lambda_k\right)^2 \left(n_l - \lambda_l\right)^2 \pi_\lambda(n) ~.
\end{equation}
Let us focus on
\begin{equation}
\sum_{n \in \mathcal{A}_\lambda^p} \left( n_i - \lambda_i \right)^2 \left(n_k - \lambda_k\right)^2 \left(n_l - \lambda_l\right)^2 \pi_\lambda(n) ~.
\end{equation}
Denoted $\overline{M}$ the cardinality of $I^p_\leq$, with $0\leq \overline{M} \leq M$, we distinguish four cases.
\begin{enumerate}[label=\arabic*)]
\item $i$, $k$, $l\in I_\geq^p$. Then,
\begin{equation}
\fl \eqalign{&\sum_{n \in \mathcal{A}_\lambda^p} \left( n_i - \lambda_i \right)^2 \left(n_k - \lambda_k\right)^2 \left(n_l - \lambda_l\right)^2 \pi_\lambda(n) \leq \cr
& \left(\prod_{q\in I_{\leq}^p} \sum_{n_q=0}^{\left\lfloor \frac{\lambda_q}{2} \right\rfloor} \frac{e^{-\lambda_q} \lambda_q^{n_q}}{n_q!}\right) \left( \prod_{r\in I_{\geq}^p}\sum_{n_r=0}^{+\infty} \left( n_i - \lambda_i \right)^2 \left(n_k - \lambda_k\right)^2 \left(n_l - \lambda_l\right)^2  \frac{e^{-\lambda_r} \lambda_r^{n_r}}{n_r!} \right) ~.}
\end{equation}
The fist term can be bounded by using Lemma  \ref{Lemma: O grande exp} and Lemma \ref{lemma equiv norme}, obtaining
\begin{equation}
\prod_{q\in I_{\leq}^p} \sum_{n_q=0}^{\left\lfloor \frac{\lambda_q}{2} \right\rfloor} \frac{e^{-\lambda_q} \lambda_q^{n_q}}{n_q!} = O \left( \Vert \lambda \Vert_{I_{\leq}^p}^{\overline{M}} e^{- c \Vert \lambda \Vert_{I^p_{\leq}}} \right)
\end{equation}
as $\Vert \lambda \Vert\to +\infty$ in $\mathcal{C}$. For the second term we consider instead the moments about the mean of a Poisson distribution (see \cite{johnson2005univariate} pag. 162) and Lemma \ref{lemma equiv norme} to have
\begin{equation}
\prod_{r\in I_{\geq}^p}\sum_{n_r=0}^{+\infty} \left( n_i - \lambda_i \right)^2 \left(n_k - \lambda_k\right)^2 \left(n_l - \lambda_l\right)^2  \frac{e^{-\lambda_r} \lambda_r^{n_r}}{n_r!} = O \left( \Vert \lambda \Vert^3_{I^p_\geq} \right)
\end{equation}
as $\Vert \lambda \Vert\to +\infty$ in $\mathcal{C}$.
\item Only two of $i$, $k$, $l$ belong to $I^p_\geq$ (in order to fix the ideas we assume $i$, $k\in I^p_\geq$ and $l\in I^p_\leq$). Then,
\begin{equation}
\fl\eqalign{&\sum_{n \in \mathcal{A}_\lambda^p} \left( n_i - \lambda_i \right)^2 \left(n_k - \lambda_k\right)^2 \left(n_l - \lambda_l\right)^2 \pi_\lambda(n) \leq \cr
& \left(\prod_{q\in I^p_\leq} \sum_{n_q=0}^{\left\lfloor \frac{\lambda_q}{2} \right\rfloor} \left(n_l - \lambda_l\right)^2 \frac{e^{-\lambda_q} \lambda_q^{n_q}}{n_q!}\right)\left( \prod_{r\in I_{\geq}^p}\sum_{n_r=0}^{+\infty} \left( n_i - \lambda_i \right)^2 \left(n_k - \lambda_k\right)^2 \frac{e^{-\lambda_r} \lambda_r^{n_r}}{n_r!} \right) ~. }
\end{equation}
We observe that $\left( n_l - \lambda_l \right)^2\leq \lambda^2_l\leq \Vert \lambda \Vert^2_{I^p_\leq}$, therefore, thanks to Lemma \ref{Lemma: O grande exp} and Lemma \ref{lemma equiv norme}, we obtain
\begin{equation}
\prod_{q\in I^p_\leq}\sum_{n_q=0}^{\left\lfloor \frac{\lambda_q}{2} \right\rfloor} \left(n_l - \lambda_l\right)^2 \frac{e^{-\lambda_q} \lambda_q^{n_q}}{n_q!} = O \left( \Vert \lambda \Vert_{I^p_\leq}^{\overline{M}+2} e^{-c \Vert \lambda \Vert_{I^p_\leq}} \right)
\end{equation}
as $\Vert \lambda \Vert\to +\infty$ in $\mathcal{C}$. Again, if we consider the moments about the mean of a Poisson distribution and Lemma \ref{lemma equiv norme}, we have
\begin{equation}
\prod_{r\in I^p_\geq}\sum_{n_r=0}^{+\infty} \left( n_i - \lambda_i \right)^2 \left(n_k - \lambda_k\right)^2 \frac{e^{-\lambda_r} \lambda_r^{n_r}}{n_r!} = O \left( \Vert \lambda \Vert^2_{I^p_\geq} \right)
\end{equation}
as $\Vert \lambda \Vert\to +\infty$ in $\mathcal{C}$.

\item Only one of $i$, $k$, $l$ belongs to $I^p_\geq$ (in order to fix the ideas we assume $i\in I^p_\geq$ and $k$, $l\in I^p_\leq$). Then we have
\begin{equation}
\fl\eqalign{&\sum_{n \in \mathcal{A}_\lambda^p} \left( n_i - \lambda_i \right)^2 \left(n_k - \lambda_k\right)^2 \left(n_l - \lambda_l\right)^2 \pi_\lambda(n) \leq \\
&\left(\prod_{q\in I^p_\leq}\sum_{n_q=0}^{\left\lfloor \frac{\lambda_q}{2} \right\rfloor} \left(n_k - \lambda_k\right)^2 \left(n_l - \lambda_l\right)^2 \frac{e^{-\lambda_q} \lambda_q^{n_q}}{n_q!}\right)\left( \prod_{r\in I^p_\geq}\sum_{n_r=0}^{+\infty} \left( n_i - \lambda_i \right)^2 \frac{e^{-\lambda_r} \lambda_r^{n_r}}{n_r!} \right) = \\
& \left(\prod_{q\in I^p_\leq}\sum_{n_q=0}^{\left\lfloor \frac{\lambda_q}{2} \right\rfloor} \left(n_k - \lambda_k\right)^2 \left(n_l - \lambda_l\right)^2 \frac{e^{-\lambda_q} \lambda_q^{n_q}}{n_q!}\right)\left( \sum_{n_i=0}^{+\infty} \left( n_i - \lambda_i \right)^2 \frac{e^{-\lambda_i} \lambda_i^{n_i}}{n_i!} \right) = \\
&\left(\prod_{q\in I^p_\leq}\sum_{n_q=0}^{\left\lfloor \frac{\lambda_q}{2} \right\rfloor} \left(n_k - \lambda_k\right)^2 \left(n_l - \lambda_l\right)^2 \frac{e^{-\lambda_q} \lambda_q^{n_q}}{n_q!}\right) \lambda_i  ~.
}
\end{equation}
We observe that $(n_k - \lambda_k)^2(n_l - \lambda_l)^2\leq \lambda_k^2 \lambda_l^2\leq \Vert \lambda \Vert_{I^p_\leq}^4$. Thanks to Lemma \ref{Lemma: O grande exp} and Lemma \ref{lemma equiv norme} we obtain
\begin{equation}
\fl\prod_{q\in I^p_\leq}\sum_{n_q=0}^{\left\lfloor \frac{\lambda_q}{2} \right\rfloor} \left(n_k - \lambda_k\right)^2 \left(n_l - \lambda_l\right)^2 \frac{e^{-\lambda_q} \lambda_q^{n_q}}{n_q!} = O \left( \Vert \lambda \Vert_{I^p_\leq}^{\overline{M}+4} e^{-c\Vert \lambda \Vert_{I^p_\leq}}  \right) ~,
\end{equation}
while trivially $\lambda_i = O(\Vert \lambda \Vert_{I^p_\geq})$ as $\Vert \lambda \Vert\to +\infty$ in $\mathcal{C}$.
\item $i$, $k$, $l\in I^p_\leq$. Then we have 
\begin{equation}
\eqalign{&\sum_{n \in \mathcal{A}_\lambda^p} \left( n_i - \lambda_i \right)^2 \left(n_k - \lambda_k\right)^2 \left(n_l - \lambda_l\right)^2 \pi_\lambda(n) \leq \cr
&\left(\prod_{q\in I^p_\leq}\sum_{n_q=0}^{\left\lfloor \frac{\lambda_q}{2} \right\rfloor} \left(n_i - \lambda_i\right)^2 \left(n_k - \lambda_k\right)^2 \left(n_l - \lambda_l\right)^2 \frac{e^{-\lambda_q} \lambda_q^{n_q}}{n_q!}\right)~.}
\end{equation}
Now, as $\left(n_i - \lambda_i\right)^2 \left(n_k - \lambda_k\right)^2 \left(n_l - \lambda_l\right)^2 \leq \lambda_i^2 \lambda_k^2 \lambda_l^2\leq \Vert \lambda \Vert_{I^p_\leq}^6 $, if we apply Lemma \ref{Lemma: O grande exp} and Lemma \ref{lemma equiv norme} we obtain
\begin{equation}
\fl\sum_{q\in I^p_\leq, n_q=0}^{\left\lfloor \frac{\lambda_q}{2} \right\rfloor} \left(n_i - \lambda_i\right)^2 \left(n_k - \lambda_k\right)^2 \left(n_l - \lambda_l\right)^2 \frac{e^{-\lambda_q} \lambda_q^{n_q}}{n_q!} = O \left( \Vert \lambda \Vert_{I^p_\leq}^{\overline{M}+6} e^{-c_4\Vert \lambda \Vert_{I^p_\leq}}  \right) 
\end{equation}
as $\Vert \lambda \Vert\to +\infty$ in $\mathcal{C}$.
\end{enumerate}
If we take into account all these facts we have
\begin{equation}
\varphi_{p,k,l}\left(\lambda\right) = O\left(\left(\Vert \lambda \Vert_{I^p_\leq}^{\overline{M}+6} e^{-c\Vert \lambda \Vert_{I^p_\leq}}\right) \frac{\Vert \lambda \Vert_{I^p_\geq}^3}{\left(\Vert \lambda \Vert_{I_\geq^p}^2 + 4\varepsilon_{k,l} \right)^2} \right)
\end{equation}
as $\Vert \lambda \Vert\to +\infty$ in $\mathcal{C}$. If $p < 2^M - 1$, then there exist at least one bit equal to $0$ in the binary decomposition of $p$ and therefore $I^p_\leq$ is not empty. If we apply Lemma \ref{lemma equiv norme}, we have that there exist $c'>0$ and a rational function $g_p$ such that
\begin{equation}
\varphi_{p,k,l}\left(\lambda\right) = O\left(g_p\left(\Vert \lambda \Vert\right) e^{-c'\Vert \lambda \Vert} \right)
\end{equation}
as $\Vert \lambda \Vert\to +\infty$ in $\mathcal{C}$, hence $\varphi_{p,k,l}$ tends to zero exponentially fast. If $p = 2^M - 1$, then $I^p_\leq = \emptyset$ and thanks to Lemma \ref{lemma equiv norme}
\begin{equation}
\varphi_{p,k,l}\left(\lambda\right) = O \left(\Vert \lambda \Vert^{-1} \right) 
\end{equation}
as $\Vert \lambda \Vert\to +\infty$ in $\mathcal{C}$. This proves \eref{Lemma: tesi 1} and therefore \eref{eq O grande 1}.
\end{proof}

\begin{proof}[Proof of Theorem \ref{KL risk theorem}]
From some simple computations we obtain
\begin{equation}
\fl\eqalign{&\mathbb{E} ( D_{\mathrm{KL}} (\lambda, \hat{\lambda}(Y) ) ) = \\ &\mathbb{E} (D_{\mathrm{KL}}(Y, \hat{\lambda}(Y))) + \mathbb{E}\left(\sum_{i=1}^M (Y_i - \lambda_i) \log \hat{\lambda}(Y) \right) - \mathbb{E} \left( \sum_{i=1}^M  Y_i\log \frac{Y_i}{\lambda_i} \right) ~.}
\end{equation}
We point out that, thanks to Lemma \ref{lemma equiv norme}, if $\Vert \lambda \Vert \to +\infty$ in the cone $\mathcal{C}$, then $\lambda_i \to +\infty$ for all $i\in\{1, \dots, M\}$. Therefore, we can apply the result proved in \cite{zanella2013corrigendum} and Lemma \ref{lemma equiv norme} to obtain

\begin{equation}
\mathbb{E}\left(\sum_{i=1}^M Y_i\log \frac{Y_i}{\lambda_i}\right) = \frac{M}{2} + O \left(\Vert \lambda \Vert^{-1} \right) \quad\mbox{}
\end{equation}
as $\Vert\lambda\Vert\to +\infty$ in $\mathcal{C}$. Finally, as a consequence of Lemma \ref{Stein's lemma (Poisson)}, we have
\begin{equation}
\mathbb{E}\left(\sum_{i=1}^M (Y_i - \lambda_i) \log \hat{\lambda}(Y) \right) = \mathbb{E}\left( (Y \nabla) \cdot \log\hat{\lambda}(Y) \right) + O\left(\Vert\lambda\Vert^{-1/2} \right) 
\end{equation}
as $\Vert \lambda \Vert\to +\infty$ in $\mathcal{C}$.
\end{proof}

\begin{proof}[Proof of Proposition \ref{Prop: propr Rk}]
We prove the thesis by induction. The case $k=1$ is trivial with $d=0$, as for all $j\in\{1, \dots, N\}$ we have
\begin{equation}
(R_1(y))_j = \sum_{i=1}^M \xi_{ji} y_i 
\end{equation}
with
\begin{equation}
\xi_{ji} \coloneqq \frac{(x_0)_j H_{ij}}{(H^T 1 )_j((Hx_0)_i + b)} > 0 ~.
\end{equation}
Let us now assume 
\begin{equation}\label{hp induttiva}
(R_k(y))_l =  \frac{p^{(d+1)}_{k,l}(y) + \dots + p^{(0)}_{k,l} }{q^{(d)}_{k,l}(y) + \dots + q^{(0)}_{k,l}}
\end{equation}
satisfying \ref{Prop rational function i} and \ref{Prop rational function ii} for all $l\in\{1, \dots, N \}$. Then, for all $i\in\{1, \dots, M \}$ we can write
\begin{equation}
(HR_k(y))_i  + b = \frac{a_{k,i}^{(Nd+1)}(y) + \dots + a_{k,i}^{(0)}}{c_{k}^{(Nd)}(y) + \dots + c_{k}^{(0)}} ~,
\end{equation}
where the summands in the numerator and in the denominator of the right-hand side are homogeneous polynomials with nonnegative coefficients. In particular, as
\begin{equation}
a_{k,i}^{(Nd+1)}(y) \coloneqq \sum_{l=1}^N H_{il} p^{(d+1)}_{k,l}(y) \left(\prod_{s=1 \atop s \neq l}^N q^{(d)}_{k,s}(y) \right) ~,
\end{equation}
\begin{equation}
c_{k}^{(Nd)}(y) \coloneqq \prod_{t=1}^N q^{(d)}_{k,t}(y) ~,
\end{equation}
and
\begin{equation}
a_{k,i}^{(0)} \coloneqq \sum_{l=1}^N H_{il} p^{(0)}_{k,l} \left(\prod_{s=1 \atop s \neq l}^N q^{(0)}_{k,s} \right) + b \left( \prod_{t = 1}^N q^{(0)}_{k,t} \right) ~,
\end{equation}
we point out that both $a_{k,i}^{(Nd+1)}$ and $c_{k}^{(Nd)}$ are complete and that $a_{k,i}^{(0)}>0$. Hence, with some simple calculations, one can prove that for all $j\in\{1, \dots, N\}$ it holds
\begin{equation}\label{update expression}
\left(\frac{1}{H^T 1}H^T\left( \frac{y}{HR_k(y)  +  b} \right)\right)_j = \frac{\tilde{p}^{(M(Nd + 1))}_{k,j}(y) + \dots + \tilde{p}_{k,j}^{(0)}}{\tilde{q}^{(M(Nd + 1))}_{k}(y) + \dots + \tilde{q}_{k}^{(0)}} ~,
\end{equation}
where the numerator and the denominator of the right-hand side consist of sums of homogeneous polynomials with nonnegative coefficients; moreover, we can observe that
\begin{equation}
\tilde{p}^{(M(Nd + 1))}_{k,j}(y) \coloneqq \frac{c_k^{(Nd)}(y)}{(H^T 1 )_j} \left(\sum_{i=1}^M H_{ij} y_i \left(\prod_{r=1 \atop r \neq i}^M a_{k,r}^{(Nd+1)}(y) \right)\right)
\end{equation}
and
\begin{equation}
\tilde{q}^{(M(Nd + 1))}_{k}(y) \coloneqq \prod_{i=1}^M a_{k,i}^{(Nd+1)}(y) 
\end{equation}
are complete and that
\begin{equation}
\tilde{q}_{k}^{(0)} \coloneqq \prod_{i=1}^M a_{k,i}^{(0)} >0 ~.
\end{equation}
If we use \eref{hp induttiva} and \eref{update expression} we obtain
\begin{equation}
(R_{k+1}(y))_j =  \frac{p^{(d+1)}_{k,j}(y) + \dots + p^{(0)}_{k,j} }{q^{(d)}_{k,j}(y) + \dots + q^{(0)}_{k,j}} \cdot \frac{\tilde{p}^{(M(Nd + 1))}_{k,j}(y) + \dots + \tilde{p}_{k,j}^{(0)}}{\tilde{q}^{(M(Nd + 1))}_{k}(y) + \dots + \tilde{q}_{k}^{(0)}} ~.
\end{equation}
Therefore we can write
\begin{equation}
(R_{k+1}(y))_j =  \frac{p^{\left(\overline{d} +1 \right)}_{k+1,j}(y) + \dots + p^{(0)}_{k+1,j} }{q^{\left(\overline{d} \right)}_{k+1,j}(y) + \dots + q^{(0)}_{k+1,j}} ~,
\end{equation}
where $\overline{d} = (MN + 1)d + M$ and 
\begin{enumerate}[label=(\roman*)]
\item $p^{\left(n\right)}_{k+1,j}$, $q^{\left(m\right)}_{k+1,j}$ are homogeneous polynomials with nonnegative coefficients for all $n\in\{1, \dots, \overline{d}+1\}$ and for all $m\in\{1, \dots, \overline{d}\}$;
\item $p^{\left(\overline{d}+1\right)}_{k+1,j}$, $q^{\left(\overline{d}\right)}_{k+1,j}$  are complete  and $q^{(0)}_{k+1,j}>0$ as
\begin{equation}
p^{\left(\overline{d}+1\right)}_{k+1,j} (y) \coloneqq p^{(d+1)}_{k,j}(y) \cdot \tilde{p}^{(M(Nd + 1))}_{k,j}(y) ~,
\end{equation}
\begin{equation}
q^{\left(\overline{d}\right)}_{k+1,j}(y) \coloneqq q^{(d)}_{k,j}(y) \cdot \tilde{q}^{(M(Nd + 1))}_{k}(y) ~,
\end{equation}
and
\begin{equation}
q^{(0)}_{k+1,j} \coloneqq q^{(0)}_{k,j} \cdot \tilde{q}_{k}^{(0)} ~.
\end{equation}
\end{enumerate}
\end{proof}

\begin{proof}[Proof of Proposition \ref{Prop: propr log lambda}]
It follows from Proposition \ref{Prop: propr Rk} that $(\hat{\lambda}_k )_i\in C^{\infty}\left(\mathbb{R}^M_+ \right)$. As $(\hat{\lambda}_k)_i$ is strictly positive on $\mathbb{R}^M_+$, then  $(\log \hat{\lambda}_k )_i$ is well defined and belongs to $C^{\infty}(\mathbb{R}^M_+)$, hence \eref{Lemma propr log: i}. Let us prove \eref{Lemma propr log: ii}. As a consequence of Proposition \ref{Prop: propr Rk}, one can easily prove that, for all $j\in\{1, \dots, M\}$
\begin{equation}
\partial_j (\log \hat{\lambda}_k )_i(y) = \frac{p^{(d)}(y) + \dots + p^{(0)}}{q^{(d+1)}(y) + \dots + q^{(0)}}
\end{equation}
with $d\in\mathbb{N}$ and $p^{(n)}$, $q^{(m)}$ homogeneous polynomials of degree $n$ and $m$ respectively for all $n\in\{1, \dots, d\}$ and for all $m\in\{1, \dots, d+1\}$. Moreover, the polynomials in the denominator have nonnegative coefficients, $q^{(d+1)}$ is complete and $q^{(0)}>0$. Since $y_i \leq \Vert y \Vert$ for all $y\in\mathbb{R}^M_+$, there exist $\tilde{p}_0, \dots, \tilde{p}_d >0$ such that 
\begin{equation}\label{eq: magg p}
\vert p^{(d)}(y) + \dots + p^{(0)}\vert \leq \tilde{p}_d \Vert y \Vert^d + \dots +  \tilde{p}_0 ~.
\end{equation}
Moreover, for Weierstrass Theorem there exists $\tilde{q}_{d+1}\geq 0$ such that
\begin{equation}
\left\vert q^{(d+1)}\left(\frac{y}{\Vert y \Vert}\right) \right\vert \geq \tilde{q}_{d+1}
\end{equation}
for all $y\in\mathbb{R}_+^M\setminus\{0\}$. The completeness of $q^{(d+1)}$ allows us to conclude that $\tilde{q}_{d+1}>0$. If $y\in\mathbb{R}^M_+\setminus\{0\}$, then
\begin{equation}
\left\vert q^{(d+1)}\left(y\right) \right\vert = \Vert y \Vert^{d+1} \left\vert q^{(d+1)}\left(\frac{y}{\Vert y \Vert}\right) \right\vert \geq \tilde{q}_{d+1} \Vert y \Vert^{d+1}
\end{equation}
(actually, the inequality $\left\vert q^{(d+1)}\left(y\right) \right\vert \geq \tilde{q}_{d+1} \Vert y \Vert^{d+1}$ holds true for all $y\in\mathbb{R}^M_+$). Therefore,
\begin{equation}\label{eq: min q}
\vert q^{(d+1)}\left(y\right) + \dots + q^{(0)} \vert \geq \tilde{q}_{d+1} \Vert y \Vert^{d+1} + q^{(0)} ~.
\end{equation}
By taking into account \eref{eq: magg p} and \eref{eq: min q} we obtain
\begin{equation}
\eqalign{& \vert \partial_j (\log \hat{\lambda}_k )_i (y)\vert \leq \frac{\tilde{p}_d \Vert y \Vert^d + \dots +  \tilde{p}_0}{\tilde{q}_{d+1} \Vert y \Vert^{d+1} + q^{(0)}} = \cr & \tilde{p}_d\frac{\Vert y \Vert^d}{\tilde{q}_{d+1} \Vert y \Vert^{d+1}+  q^{(0)}} + \dots + \frac{\tilde{p}^{(0)}}{\tilde{q}_{d+1} \Vert y \Vert^{d+1}+  q^{(0)}} ~.}
\end{equation}
Hence, \eref{Lemma propr log: ii} is straightforward from the fact that, for all $s\in\{0, \dots, d\}$ and for all $c>0$ there exist $a_s$, $b_s>0$ such that 
\begin{equation}
\frac{t^s}{t^{d+1} + c} \leq \frac{a_s}{t + b_s}
\end{equation}
for all $t\geq 0$. The proof of \eref{Lemma propr log: iii} is analogous to the one of \eref{Lemma propr log: ii}.
\end{proof}

\newcommand{\newblock}{}
\bibliographystyle{apa}
\bibliography{mybibliography}

\end{document}